\documentclass[a4paper,10pt]{amsart}
\usepackage{graphicx,verbatim}
\usepackage{amssymb}
\usepackage{amsthm}
\usepackage{amsmath}
\usepackage{epstopdf}
\usepackage{textgreek}
\usepackage{amsfonts}
\usepackage{mathrsfs}
\usepackage[english]{babel}
\usepackage{lscape} 
\usepackage[vcentermath,enableskew]{youngtab}
\usepackage{tikz,tikz-cd}
\usetikzlibrary{arrows,positioning,automata,shadows,fit,shapes}

\newtheorem{theorem}{Theorem}[section]
\newtheorem{definition}[theorem]{Definition}
\newtheorem{example}[theorem]{Example}
\newtheorem{corollary}[theorem]{Corollary}
\newtheorem{lemma}[theorem]{Lemma}

\newtheorem{remark}[theorem]{Remark}
\newtheorem*{theorem*}{Theorem}
\newtheorem*{definition*}{Definition}
\newtheorem*{lemma*}{Lemma}

\DeclareMathOperator{\irr}{Irr}
\DeclareMathOperator{\triv}{triv}
\DeclareMathOperator{\Pin}{Pin}
\DeclareMathOperator{\Det}{Det}
\DeclareMathOperator{\Dim}{Dim}
\DeclareMathOperator{\Ker}{Ker}

\DeclareMathOperator{\im}{im}

\DeclareMathOperator{\sgn}{sgn}

\DeclareMathOperator{\stab}{stab}

\DeclareMathOperator{\Spec}{Spec}

\DeclareMathOperator{\Span}{span}

\title[Dirac cohomology of the Dunkl-Opdam subalgebra]{Dirac cohomology of the Dunkl-Opdam subalgebra via inherited  Drinfeld properties}
\date{\today}
\author{Kieran Calvert}

\begin{document}
\maketitle
\setcounter{tocdepth}{1} 
\setcounter{secnumdepth}{3}
\tableofcontents

\begin{abstract} 
In this paper we define a new presentation for the Dunkl-Opdam subalgebra of the rational Cherednik algebra. This presentation uncovers  the Dunkl-Opdam subalgebra as a Drinfeld algebra. We  use this fact to define Dirac cohomology for the DO subalgebra. We also formalise generalised graded Hecke algebras and extend a Langlands classification to generalised graded Hecke algebras.\end{abstract}

\begin{section}{Introduction}

We study the Dunkl-Opdam subalgebra, $\mathbb{H}_{DO}$, of the rational Cherednik algebra associated to $G(m,1,n) = S_n \rtimes (\mathbb{Z}_m)^n$, introduced by Dunkl and Opdam \cite{DO03}. This subalgebra of the rational Cherednik algebra $\mathcal{H}_t(G(m,1,n))$  (Definition \ref{rationalcherednikdef}) algebra is independent of the parameter $t$. In this chapter we take a closer look at $\mathbb{H}_{DO}$ and notice that it is similar to both graded Hecke algebras and Drinfeld algebras. We extend several results for Hecke algebras and faithful Drinfeld algebras to include the Dunkl-Opdam subalgebra. 
We construct a new presentation of $\mathbb{H}_{DO}$:
\begin{theorem*}\label{Drinfeldpresentation} There exists a  presentation of $\mathbb{H}_{DO}$ given by elements $\{\tilde{z}_i: i=1,...,n\}$ and elements in $G$ such that:
$$s_i \tilde{z}_j s_i^{-1} = s_i(\tilde{z}_j) ,$$
$$g_i \tilde{z}_j = \tilde{z}_j g_i \hspace{1cm} \forall i,j = 1,...,n,$$
$$ [\tilde{z}_i,\tilde{z}_j] \in \mathbb{C}G.$$
\end{theorem*} 
This presentation exposes $\mathbb{H}_{DO}$ as a Drinfeld algebra. Drinfeld \cite{D86} initially defined these algebras (Definition \ref{drinfeldalgdef})  with the potential to have non faithful representations. In the literature this has been largely forgotten, perhaps because there appeared to be no natural examples of a non-faithful Drinfeld algebra. The Dunkl-Opdam subalgebra is a naturally occurring non-faithful Drinfeld algebra. Ciubotaru \cite{C16} defined Dirac cohomology for faithful Drinfeld algebras and we extend this to non-faithful Drinfeld algebras. 

Dez\'el\'ee introduced the idea of generalised graded Hecke algebra to look at the Dunkl-Opdam subalgebra. We concretely define a class of generalised graded Hecke algebra, which contains Dez\'el\'ee's examples. We extend Evans' \cite{E96} Langlands classification to generalised graded Hecke algebras. 

\begin{theorem*} Let $\mathbb{GH}$ denote a generalised graded Hecke algebra. A parabolic subalgebra is denoted by $\mathbb{GH}_P$, with $\mathbb{GH}_{P_S}$ denoting the semisimple part of $\mathbb{GH}_P$ (Definition \ref{parabolic}).

(i) Every irreducible $\mathbb{GH}$ module $V$ can be realised as a quotient of \\
$\mathbb{GH}(W\rtimes T) \otimes_{\mathbb{GH}_P} U$, where $U = \hat{U} \otimes \mathbb{C}_\nu$ is such that $\hat{U}$ is an irreducible tempered $\mathbb{GH}_{P_s}$ module and $\mathbb{C}_\nu$ is a character of $S(\mathfrak{a})$ defined by $\nu \in \mathfrak{a}^{*^+}$. 

(ii) If U is as in (i) then $\mathbb{H}(W\rtimes T) \otimes_{\mathbb{GH}_P} U$ has a unique irreducible quotient to be denoted $J(P,U)$.

(iii) If $J(P,\hat{U}\otimes \mathbb{C}_\nu) \cong J(P',\hat{U}'\otimes\mathbb{C}_{\nu'})$ then $P = P'$, $\hat{U} \cong \hat{U}'$ as $\mathbb{GH}_{P_s}$ modules and $\nu = \nu'$.

\end{theorem*}
The Dunkl-Opdam subalgebra has a commutative subalgebra $  \mathbb{C}T \cong (\mathbb{Z}_m)^n$. We decompose representations into weight spaces (Definition \ref{weightspacedef}), which then defines weights of a representation. The weights of a $\mathbb{H}_{DO}$ representation come in orbits (Lemma \ref{lang2}). We use these weights to highlight that irreducible representations of $\mathbb{H}_{DO}$ are pullbacks of $\mathbb{H}(S_{a_i})$ representations via specific quotients (Lemma  \ref{quotientspullback}).    Define the set $A = \{ \underline{a}\in \mathbb{N}^m:\sum a_i=n\}$, there is a  Morita equivalence 
$$\mathbb{H}_{DO}(G(m,1,n)) \xrightarrow{\text{Morita}}\bigoplus_{\underline{a}\in A} \mathbb{H}(S_{a_0})\otimes ...\otimes\mathbb{H}(S_{a_{m-1}}) \text{\hspace{1cm}(Theorem \ref{moritaequiv}).}$$

We use this Morita equivalence to describe the Dirac cohomology of a $\mathbb{H}_{DO}$ module $X$ in terms of Dirac cohomology of the associated $\mathbb{H}(S_{a_i})$ modules. 
Let $F$ and $F^{-1}$ be functors displaying this equivalence.
\begin{theorem*} Given an irreducible representation $V$ with $\mathbb{C}[T]$ weight space $F(V) =V_{\mu_{\underline{a}}}$, where $\underline{a}=(a_0,...,a_{m-1})$. The space $F(V)$ is as a $\mathbb{H}_{S_{a_0}} \otimes...\otimes \mathbb{H}_{S_{a_{m-1}}}$ module $F(V) \cong X_{a_0} \otimes...\otimes X_{a_{m-1}}$. The Dirac cohomology of $V$ is 
$$\bigoplus_{c\in S_P/S_n} c \left (  H_D(X_{a_0}) \otimes ...\otimes H_D(X_{a_{m-1}})\right ),$$
where $H_D(X)$ is the type A Dirac cohomology of the $\mathbb{H}_{S_k}$-module $X$. 
Let $H_D(\bullet)$ denote the functor taking the relevant module to  its Dirac cohomology. We have the following commutative diagram:

\begin{centering}

\begin{tikzcd}
\mathbb{H}(G(m,1,n))\text{-mod}\arrow[r, "F"] \arrow[d, "H_D(\bullet)"] & \bigoplus_{\underline{a}\in A}\mathbb{H}_{S_{a_0}} \otimes...\otimes \mathbb{H}_{S_{a_{m-1}}} \text{-mod}\arrow[d, "H_D(\bullet)" ] \\\mathbb{C}\widetilde{G(m,1,n)}\text{-mod} 
& \bigoplus_{\underline{a}\in A} \mathbb{C}\widetilde{S_{a_0}} \otimes ...\otimes \mathbb{C}\widetilde{S_{a_{m-1}}}\text{-mod}\arrow[l, "\widetilde{F^{-1}}"]
\end{tikzcd}
\end{centering}
\end{theorem*}

In Section \ref{drinfeldalgebras} we study Drinfeld algebras, we focus on the fact that these algebras can be defined with a non-faithful representation. We extend Dirac cohomology defined in \cite{C16} to the non-faithful case. In Section \ref{generalisedgradedheckealg}, we introduce the class of generalised graded Hecke algebras and we extend Evens' \cite{E96} Langlands classification to this class. In Section \ref{DOsubalg} we introduce the Dunkl-Opdam subalgebra. We highlight that it is a generalised graded Hecke algebra and by introducing a new presentation show that it is also a non-faithful Drinfeld algebra. Section \ref{reptheory} defines a Morita equivalence between the DO subalgebra and direct sums of graded Hecke algebras associated to  parabolic symmetric groups. In Section \ref{diraccohofDO} we combine results on Dirac cohomology (Section \ref{drinfeldalgebras}) and the Morita equivalence (Section \ref{reptheory}) to describe Dirac cohomology of a $\mathbb{H}_{DO}$ module with Dirac cohomology of its associated module under the Morita equivalence. This highlights that the Morita equivalence behaves well with respect to Dirac cohomology. 

\end{section}

\begin{section}{Drinfeld algebras}\label{drinfeldalgebras}
 In this section we will define Drinfeld algebras as introduced by Drinfeld. Ciubotaru \cite{C16} defined Dirac cohomology for faithful Drinfeld algebras and we will extend Dirac cohomology to non-faithful Drinfeld algebras. The results in this section follow almost verbatim from the proofs in \cite{C16} so we will not write them here.

 Given a finite group $G$, antisymmetric bilinear forms $b_g$ for $g \in G$ and a representation $(\rho, V)$ of G, then we construct an algebra
 $$ \mathbb{H} =\mathbb{C}[G]\rtimes T(V)/R.$$
 Here $R$ is the two sided ideal of $\mathbb{C}[G]\rtimes T(V)$ generated by the relations,
 $$g^{-1}vg = \rho(g)(v) \text{ for all } g\in G \text{ and } v\in V,$$
 and
 $$[u,v] = \sum_{g \in G} b_g(u,v) g \text{ for all } v,u \in V.$$
 
We define a filtration on the algebra $\mathbb{C}[g]\rtimes T(V)/R$, a vector $v$ has degree $1$ and a group element $g \in G$ has degree $0$. 
 \begin{definition}\cite{D86} \label{drinfeldalgdef} An algebra of the form $\mathbb{H} =\mathbb{C}[G]\rtimes T(V)/R$ is a Drinfeld algebra if it satisfies a PBW criterion. That is the associated graded algebra is naturally isomorphic to 
 $$\mathbb{C}[G] \rtimes S(V).$$
Here $\rtimes$ denotes the semi direct product with the natural action of $G$ on $V$. \end{definition}
 We state the conditions on the bilinear forms $b_g$ such that $\mathbb{H}$ is a Drinfeld algebra. This was originally stated in \cite{D86}, and explained for the faithful case in \cite{RS03}.
 Define $G(b) = \{g\in G : b_g \neq 0\}$.
 \begin{theorem}\cite{D86}\cite[Theorem 1.9]{RS03} The algebra $\mathbb{H}$ is a Drinfeld algebra if and only if:

\indent(1) For every $g \in G$, $b_{g^{-1}hg} (u,v) = b_h (\rho(g)(u),\rho(g)(v))$ for every $u,v \in V,$

\indent(2) For every $g \in G(b) \setminus \Ker\rho$, then $\Ker b_g = V^{\rho(g)}$ and $\Dim(V^{\rho(g)}) = \Dim V - 2,$

 \indent(3) For every $g \in G(b) \setminus \Ker\rho$ and $h \in Z_G(g),$ $\Det(h|_{V^{\rho(g)^\perp}}) = 1$, where $V^{\rho(g)^\perp }= \{ v - \rho(g)(v): v\in V\}$.

  \end{theorem}
The above statements follows immediately from the proofs given in \cite{RS03} for the faithful case. The only variation is that $1$ in \cite{RS03} is replaced by the set $\Ker \rho$.

\begin{subsection}{Non-faithful Drinfeld algebras}

  In the recent literature, Drinfeld algebras have predominately been considered with $G$ a subgroup of $GL(V)$, however Drinfeld originally expressed them with a potentially non-faithful representation. To address this disparity and to avoid confusion we will say that a Drinfeld algebra is a faithful Drinfeld algebra if the representation involved is faithful and we will say that a Drinfeld algebra is non-faithful if the representation is non-faithful. The class of Drinfeld algebras includes both faithful and non-faithful Drinfeld algebras.

\end{subsection}
\begin{subsection}{The Dirac operator for (non-faithful) Drinfeld algebras}\label{Dirac}

If $V$ has a $G$-invariant symmetric bilinear form then one can define a Dirac operator $\mathcal{D}$. In \cite{C16} Dirac cohomology is defined for any faithful Drinfeld algebra. Furthermore an equation involving the square of the Dirac operator is proved \cite[Theorem 2.7]{CT11}. The extension of these theorems to the case of non-faithful representations is clear from the proofs. We will however give the equivalent formulation of the theorems in the non-faithful case.  In this section we will denote a Drinfeld algebra by $\mathbb{H}$.

\begin{subsubsection}{The Clifford algebra}  Let $\langle , \rangle$ be a $G$-invariant non-degenerate bilinear form on $V$. The Clifford algebra $C(V)$ associated to $V$ and $\langle , \rangle$ is the quotient of the tensor algebra $T(V)$ by the relations
$$v \cdot v' + v' \cdot v = - 2\langle v ,v'\rangle.$$
 The Clifford algebra has a filtration by degrees and a $\mathbb{Z}/2\mathbb{Z}$-grading by parity of degrees. In this grading $C(V) = C(V)_0 \oplus C(V)_1$. We define an automorphism $\epsilon: C(V) \to C(V)$ which is the identity on $C(V)_0$ and minus the identity on $C(V)_1$. Let us extend $\epsilon$ to be an automorphism of $\mathbb{H}\otimes C(V)$ by defining $\epsilon$ be the identity on $\mathbb{H}$. We define an anti-automorphism, the transpose of $C(V)$, a antihomomorphism such that, $v^t = -v$ for all $v \in V$. The Pin group is: 
$$\Pin(V) = \{ a\in C(V)^\times : \epsilon(a)  V  a^{-1} \subset V, a^t = a^{-1}\}.$$
Let $(\rho,V)$ be a representation of $G$ with $G$-invariant form. This establishes $\rho(G)$ as a subgroup of $O(V)$.
The Pin group is a double cover of $O(V)$ with surjection $p:\Pin(V) \to O(V)$. We define the pin double cover of $\rho(G) \subset O(V)$ as 
$$\widetilde{\rho(G)} := p^{-1}(\rho G) \subset \Pin (V).$$
Note that $\widetilde{\rho(G)}$ is not a double cover of $G$ but it is a double cover of $\rho(G)$. We construct a cover of $G$. We will define $\tilde{G}$ to be the semi direct product $ \Ker \rho \rtimes \widetilde{\rho(G)} $ with cross multiplication:
$$(h,\tilde{g})\cdot (h',\tilde{g}') = (hg^{-1}h'g, \tilde{g}\tilde{g}'), \hspace{1cm} \text{ for all } \tilde{g},\tilde{g}' \in \widetilde{\rho(G)} \text{ and } h, h' \in \Ker\rho.$$
Given $\tilde{G}$ we can embed it in $\mathbb{H} \otimes C(V)$ via 
$$\Delta: \tilde{G} \to \mathbb{H} \otimes C(V),$$
$$\Delta(\tilde{g}, h) = hp(\tilde{g})\otimes \tilde{g}, \hspace{1cm} \tilde{g} \in \widetilde{\rho(G)}, h \in \Ker\rho.$$
For more information on the Clifford algebra see \cite{HP06} and \cite{M13}.
\end{subsubsection}
\begin{subsubsection}{The Dirac element} 
Given any basis $\{v_i\}$ of $V$ and dual basis $\{v^i\}$ with respect to $\langle , \rangle$ we define the Dirac element
$$\mathcal{D} = \sum_{i} v_i \otimes v^i \in \mathbb{H} \otimes C(V).$$
 We give a formula for $\mathcal{D}^2$. This is equivalent to \cite[Theorem 2.7]{C16},. The only variation being that $\ker \rho$ replaces $1$.

For every $g\in G(b)$ set, $$\textbf{k}_g = \sum_{i,j} b_g(v_i,v^j)v^iv_j \in C(V),$$
and $$\textbf{h} = \sum_i v_iv^i \in \mathbb{H}.$$
The commutation relation defined for a Drinfeld algebra shows:
$$\mathcal{D}^2 = -\textbf{h}\otimes 1 +\frac{1}{2} \sum_{g \in G(b)} g \otimes \textbf{k}_g.$$
This result is \cite[Lemma 2.5]{C16}.
Recall $G(b)  = \{g\in G : b_g \neq 0\}$, we write $\widetilde{G}(b)$ for the cover of this subset.

\begin{lemma}\label{productofalphabeta} Similarly to \cite[Lemma 2.6]{C16} every element $g$ in $G(b) /\Ker \rho$ can be expressed a a product of two reflections. Every element in $G(b) \setminus\Ker \rho $ can be written as an coset representative of $G(b)/\Ker \rho$ conjugated by an element in $\Ker \rho$. Therefore given $g \in G(b) \setminus \Ker \rho$, there exists an $h \in \Ker \rho$ and $\alpha,\beta \in V$ such that $g = h^{-1} s_\alpha s_\beta h$ and the roots $\alpha, \beta $ span the space $ (V^{\rho(g)})^\perp$. We scale $\alpha$ and $\beta$ such that  $\langle \alpha ,\alpha \rangle = \langle \beta ,\beta \rangle = 1$.\end{lemma}
\begin{proof} See proof of \cite[Lemma 2.6]{C16} with $G$ replaced by $G/ \Ker(\rho)$. \end{proof}
For every coset representative $g \in G(b)/\Ker(\rho)$ define
$$\tilde{g} = \alpha\beta \in C(V), \hspace{1cm} c_{\tilde{g}} = \frac{b_g(\alpha,\beta)}{1 - \langle \alpha ,\beta\rangle^2}\in \mathbb{C}, \hspace{1cm} e_{g} = \frac{b_g(\alpha,\beta)\langle \alpha \beta \rangle}{1 - \langle \alpha ,\beta\rangle^2}\in \mathbb{C}.$$
Every $x\in \widetilde{G}(b)$ can be written as $h^{-1} g h$ where $g$ is a coset representative of $\tilde{G}(b)/ \Ker \rho$ and $h \in \Ker \rho$. Lemma \ref{productofalphabeta} gives $g = s_\alpha s_\beta$ and $\tilde{g} = \alpha\beta \in C(V)$.
We define, for $x = h^{-1}gh \in \tilde{G}$,
$$\tilde{x} = \tilde{g} = \alpha \beta \in C(V), \hspace{1cm} c_{\tilde{x}} = c_{\tilde{g}_C}, \hspace{1cm} e_{x} = he_{g_C}h^{-1}.$$

Let us define the Casimir elements, $\Omega_\mathbb{H}$ in $\mathbb{H}$ and $\Omega_{\tilde{G}}$ in $\tilde{G}$.
$$\Omega_{\mathbb{H}} = \textbf{h} - \sum_{g \in G(b) / \Ker\rho} e_g g \in \mathbb{H}^G,$$
$$\Omega_{\tilde{G}} =   \sum_{\substack{h\in \Ker\rho\\ g \in G(b) /\Ker \rho}}  h^{-1}\tilde{g}h c_{\tilde{g}}\in \mathbb{C}[\tilde{G}]^{\tilde{G}}.$$

\begin{theorem} \cite[c.f. Theorem 2.7]{C16} The square of the Dirac element can be expressed as a sum of the two Casimir elements plus terms from the kernel;
$$\mathcal{D}^2 = -\Omega_\mathbb{H} \otimes 1 +\Delta(\Omega_{\tilde{G}}) +  \frac{1}{2} \otimes \sum_{g \in \ker \rho} \textbf{k}_g.$$ 
\end{theorem}

\end{subsubsection}
\begin{subsubsection}{Vogan's Morphism}
Let $\tilde{\Omega}_{\mathbb{H}} = \Omega_{\mathbb{H}} - \frac{1}{2} \otimes \sum_{h \in \Ker\rho} \textbf{k}_h$, define
$$\textbf{A} = Z_{\mathbb{H}\otimes C(V)} (\tilde{\Omega}_{\mathbb{H}}) \subset (\mathbb{H} \otimes C(V))^{\tilde{G}}.$$
If $\Ker\rho \cap G(b) = \emptyset$ then $\textbf{A} = \mathbb{H} \otimes C(V)$.
Define a derivation
$$d:\mathbb{H} \otimes C(V) \to \mathbb{H} \otimes C(V),$$
$$d(a) = \mathcal{D}a - \epsilon(a) \mathcal{D}.$$
The Dirac operator $\mathcal{D}$ interchanges the trivial and $\Det \tilde{G}$ - isotypic spaces of $\textbf{A}$. We define $d_{\triv}$ and $d_{\Det}$ to be the restriction of $d$ to the trivial and $\Det \tilde{G}$ - isotypic spaces. 
We state the theorems in \cite{C16} but note that the proofs apply verbatim to this case.
\begin{theorem}\label{directsummand}\cite[c.f. Theorem 3.5]{C16}The kernel of $d_{\triv}$ equals:
$$\Ker d_{\triv} = \im d_{\Det} \oplus \Delta (\mathbb{C}[\tilde{G}]^{\tilde{G}}).$$\end{theorem}
Since  $d$ is a derivation then $\Ker d_{\triv}$ is an algebra. The following theorem is the statement of Vogan's Dirac homomorphism in the non-faithful Drinfeld case
\begin{theorem} \cite[c.f. Theorem 3.8]{C16} The projection $\zeta: \Ker d_{\triv} \to \mathbb{C}[\tilde{G}]^{\tilde{G}}$ defined in Theorem \ref{directsummand} is an algebra homomorphism. \end{theorem}
Note that since the image of $\zeta$ is an abelian algebra the morphism must factor through the abelianisation of $\Ker d_{\triv}$. 
Recall that when $b_{g} = 0$ for all $g \in \Ker \rho$ then $Z(\mathbb{H}) \otimes 1$ is contained in $\Ker d_{\triv}$. With this extra condition we can consider the dual of $\zeta$ which relates the representations of $\tilde{G}$ with characters of $Z(\mathbb{H})$. 
$$\zeta^*: \irr(\tilde{G}) = \Spec \mathbb{C}[\tilde{G}]^{\tilde{G}} \to \Spec Z(\mathbb{H}).$$
Here $\Spec$ denoted the algebra of characters on a given algebra.

\end{subsubsection}
\end{subsection}
\end{section}

\begin{section}{Generalised graded Hecke algebras}\label{generalisedgradedheckealg}

Ram and Shepler \cite{RS03} show that there does not exist a faithful Drinfeld algebra associated to the complex reflection group $G(m,1,n) = S_n \rtimes (\mathbb{Z}_m)^n$. However they define a candidate for an algebra that is similar to a graded Hecke algebra. Dez\'el\'ee \cite{D06} introduced the term generalized graded Hecke algebra. In Section \ref{DOsubalg} we show that these algebras are non-faithful Drinfeld algebras. We define a larger group of algebras denoted generalised graded Hecke algebras or GGH for short.

Set $A \star B$ to be the free product of unital associative complex algebras.
\begin{definition}\label{GGHdef} Let $W$ be a Weyl group generated by simple reflections $s_\alpha, \alpha \in \Pi$. $W$ acts on a commutative group $T$,  $\mathfrak{t}$ is a faithful complex  $W$-representation. $\langle,\rangle$ is a $W$-invariant pairing between the vector spaces $\mathfrak{t}^*$ and $\mathfrak{t}$. We define a parameter function 
$$\tilde{c}: \Pi \to \mathbb{C}[T].$$
The generalised graded Hecke algebra $\mathbb{GH}(W\rtimes T)$ is the quotient of the algebra 
$$ \mathbb{C} [W \rtimes T] \star S(\mathfrak{t})$$
by the relations
$$s_\alpha t = s_\alpha(t) s_\alpha + \langle \alpha ,t \rangle \tilde{c}(\alpha), \hspace{1cm} \forall t \in \mathfrak{t}, \alpha \in \Pi$$
$$[h,t] =0 \hspace{1cm} \forall t \in \mathfrak{t}, h \in T.$$

\end{definition}
In the case that $T$ is the trivial group this includes all graded Hecke algebra. In this form the relations look very similar to the graded Hecke algebras except that the parameter function takes values in $\mathbb{C}[T]$ instead of $\mathbb{C}$. 
In the following section we will prove a Langlands classification for generalised graded Hecke algebras. This follows Evens' \cite{E96} proof of the Langlands classification for graded Hecke algebras.

\begin{subsection}{Preliminaries for the Langlands classification}

Let $\{X,R,Y,\check{R}, \Pi\}$ be root datum, where $X$ and $Y$ are free finitely generated abelian groups and there exists a perfect bilinear pairing between them. The roots $R \subset X$ and coroots $\check{R} \subset Y$ are finite subsets with a bijection between them.  let $\Pi$ denote the simple roots $\{\alpha_1,...,\alpha_{l}\}$. Positive roots $R_+$ (resp. $\check{R}^+$) are the $\mathbb{N}$ span of $\Pi$ (respectively $\check{\alpha}$ for $\alpha \in \Pi$).  
Let $\mathfrak{t} = X \otimes \mathbb{C}$ and $\mathfrak{t}^* = Y \otimes \mathbb{C}$ be dual vector spaces, similarly let $\mathfrak{t}_\mathbb{R}, \mathfrak{t}_\mathbb{R}^*$ be the real spans of $X$ and $Y$. 
Let $T$ be a finite abelian group such that $W$ acts on $T$. Let  $\tilde{c}$ be a function from $\Pi$ to $\mathbb{C}T$ which is constant on conjugacy classes.

\begin{definition}\label{GGH}  The generalised graded Hecke algebra associated to the root system of $W$, $T$ and $\tilde{c}$ is defined to be the free product of algebras
$$\mathbb{GH}(W \rtimes T) \cong \mathbb{C}[W\rtimes T] \star S(\mathfrak{t}).$$
Modulo the relations:
$$s_\alpha x - s_\alpha(x) s_\alpha =\langle \alpha, x \rangle \tilde{c}(\alpha), \hspace{1cm} \forall x \in \mathfrak{t},  \alpha \in \Pi$$
$$g x = xg \hspace{1cm} \forall x \in  \mathfrak{t}, g \in T,$$
and the requirement that $\mathbb{C}[W\rtimes T]$ and $S(\mathfrak{t})$ are subalgebra 

 \end{definition} 
In the case of the generalised graded Hecke algebra associated to $G(m,1,n)$, we set $W= S_n$, $T =( \mathbb{Z}_m)^n$.
The function $\tilde{c}$ from $\Pi$ to $\mathbb{C}T$ is defined by $\tilde{c}( \epsilon_{i+1} - \epsilon_i ) = \sum_{l=0}^{m-1} g_i^lg_{i+1}^{-l}.$ The element $g_i$ is the $i^{th}$ generator of $(\mathbb{Z}_m)^n$.
We will denote a generalised graded Hecke algebra by $\mathbb{GH}$. 

\begin{definition}\label{parabolic} Given a subset $\Pi_P$ of $\Pi$ we can define a parabolic subgroup $W_P$ of $W$ generated by $s_\alpha$ for $\alpha \in \Pi_P$. The corresponding parabolic subalgebra of the generalised graded Hecke algebra $\mathbb{GH}_P$ is generated by $s_\alpha$ for $\alpha \in \Pi$ and $\mathbb{C}[T] \otimes S(\mathfrak{t})$. This is the generalised graded Hecke algebra associated to $W_p \rtimes  T$. \end{definition}

\begin{definition} Define $\mathfrak{a}$  to be the vector space $\{x\in \mathfrak{t} :  \check{\alpha}(x) = 0, \alpha \in \Pi\}$. Given any parabolic subalgebra $\mathbb{GH}_P$, set $\mathfrak{a}_P = \{x\in \mathfrak{t} :  \check{\alpha}(x) = 0, \alpha \in \Pi_P\}$.  Let $\mathfrak{a}_s$ be the perpendicular subspace to $\mathfrak{a}_P^*$ under the pairing  of $\mathfrak{t}$ and $ \mathfrak{t}^*$. \end{definition}

Then $\mathbb{GH}_P \cong \mathbb{GH}_{P_s}(W_P\rtimes T) \otimes S(\mathfrak{a}_P)$, where $\mathbb{GH}_{P_s}(W_P\rtimes T)$ is constructed (Definition \ref{GGH}) as the quotient of the algebra: 
$$\mathbb{GH}_{P_s}(W_P\rtimes T) \cong \mathbb{C}[W\rtimes T] \star S(\mathfrak{a}_s).$$

The commutative subalgebra $\mathscr{A} = \mathbb{C}T \otimes S(\mathfrak{t})$ features in all parabolic subalgebras. For every $\mathscr{A}$ module $V$ we can consider a weight space decomposition. 

\begin{definition}\label{weightspacedef} Let $\mathscr{A}= \mathbb{C}T \otimes S(\mathfrak{t})$ and $\mathscr{A}^*$ denote characters on this algebra. Given an $\mathscr{A}$ module $V$ and character $\mu \otimes \lambda \in\mathscr{A}^* =\mathbb{C}T^* \otimes S(\mathfrak{t})^*$ define the subspace: 
$$V_{\mu\otimes \lambda} = \{v \in V : y \otimes x (v) = \mu(y) \otimes \lambda(x) (v) \text{ for all } y \otimes x \in \mathscr{A}\}.$$ 
We can decompose $V$ into weight spaces:
$$V = \bigoplus_{\lambda \otimes \mu \in \mathscr{A}^*} V_{\mu \otimes \lambda}.$$
The weights of $V$ are the $\mu \otimes \lambda \in \mathscr{A}^*$ such that $V_{\mu\otimes \lambda}$ is non zero. \end{definition}

\begin{definition} Given simple roots $\alpha_1,...,\alpha_n$. The fundamental coweights $x_i \in \mathfrak{t}$ are such that 
$$\check{\alpha}_j(x_i) = \delta_{ij} \text{ and } \nu(x_i) = 0 \text{ for all } \nu \in \mathfrak{a}^*.$$
\end{definition}

\begin{example} Let $W = S_n$, $\mathfrak{t} = \Span\{\epsilon_1,...,\epsilon_n\}$, $\mathfrak{t}^* = \Span\{e_1,...,e_n\}$. $\mathfrak{a}^* = \Span\{a=e_1+e_2 +...+e_n\}.$ Let the simple roots be $\alpha_i = \epsilon_i - \epsilon_{i+1}$ for $i = 1,...,n-1$. Then the fundamental coweights are
$$x_i = \sum_{j\leq i}\epsilon_j -\frac{i}{n}\left(\epsilon_1+\epsilon_2+...+\epsilon_n\right).$$
The modification by $\epsilon_1+...+\epsilon_n$ is required so that $\mathfrak{a}^*$ is perpendicular to $x_i$. If we had defined $\mathfrak{t}$ to be $\Span\{\epsilon_1,...,\epsilon_n\}/ \Span\{\epsilon_1+...+\epsilon_n\}$ this would not be required as $\mathfrak{a}^* = 0$. 
\end{example}   

\begin{definition} An irreducible $\mathbb{GH}$ module $V$ is essentially tempered if for all weights $\mu \otimes \lambda: \mathbb{C}T \otimes S(\mathfrak{t}) \to \mathbb{C}$ of $V$, $Re(\lambda(x_i)) \leq 0$, for all fundamental coweights $x_i$. The module $V$ is tempered if $V$ is essentially tempered and $Re(\lambda|_{\mathfrak{a}_\mathbb{R}}) = 0$. Here $\mathfrak{a}_\mathbb{R}$ is the real span of $x \in X$ perpendicular to the coroots. \end{definition} 

Let $$\mathfrak{a}_P^{*^+} = \{\nu \in \mathfrak{a}_P^* : Re(\nu(\alpha)) > 0, \alpha \in \Pi - \Pi_P\}.$$

\end{subsection}

\begin{subsection}{The Langlands classification for generalized graded Hecke algebras}\label{Langlandssection}
\begin{theorem}

(i) Every irreducible $\mathbb{GH}$ module $V$ can be realised as a quotient of $\mathbb{GH}(W\rtimes T) \otimes_{\mathbb{GH}_P} U$, where $U = \hat{U} \otimes \mathbb{C}_\nu$ is such that $\hat{U}$ is an irreducible tempered $\mathbb{GH}_{P_s}$ module and $\mathbb{C}_\nu$ is a character of $S(\mathfrak{a}_P)$ defined by $\nu \in \mathfrak{a}_P^{*^+}$. 

(ii) If U is as in (i) then $\mathbb{H}(W\rtimes T) \otimes_{\mathbb{GH}_P} U$ has a unique irreducible quotient to be denoted $J(P,U)$.

(iii) If $J(P,\hat{U}\otimes \mathbb{C}_\nu) \cong J(P',\hat{U}'\otimes\mathbb{C}_{\nu'})$ then $P = P'$, $\hat{U} \cong \hat{U}'$ as $\mathbb{GH}_{P_s}$ modules and $\nu = \nu'$.

\end{theorem}
First we state a couple of lemmas of Langlands and a technical lemma about orbits of weights. 
Let $Z$ be a real inner product space of dimension $n$. Let $\{\check{\alpha}_1,...,\check{\alpha}_n\}$ be a basis such that $(\check{\alpha}_i,\check{\alpha}_j) \leq 0$ whenever $i \neq j$. Let $\{ \beta_1,..,\beta_n\}$ be a dual basis. For a subset $F$ of $\Pi$, let 
$$S_F = \{ \sum_{j\notin F} c_j \beta_j - \sum_{i\in F} d_i \check{\alpha}_j : c_j > 0 ,d_i \leq 0\}.$$

\begin{lemma}\cite[IV, 6.11]{BW80} Let $x \in Z$. Then $x \in S_F$ for a unique subset $F = F(x)$. \end{lemma} 
If $x\in Z$ then let $x_0 = \sum_{j \notin F}c_j\beta_j$, where $x \in S_F$ and $x = \sum_{j\notin F} c_j \beta_j - \sum_{i\in F} d_i \check{\alpha}_i$. It is clear that if $x_0 = y_0$ then $F(x) = F(y)$. Define a partial order on $Z$ by setting $x \geq y$ if  $x - y = \sum_{t_i \geq 0} t_i\check{\alpha}_i$. 

\begin{lemma}\cite[IV, 6.13]{BW80} \label{lang2} If $x,y \in Z$ and $x \geq y$ then $x_0 \geq y_0$. \end{lemma} 
\begin{lemma}\label{orbit} Given an irreducible $\mathbb{GH}_P$  module $V$, the set of weights $\{\lambda\otimes \mu\}$ are all in the same $W_P$ orbit. \end{lemma}
\begin{proof}The group $W_P$ is the only part of $\mathbb{GH}_P$ which does not act by eigenvalues on $V_{\mu \otimes \lambda}$. For any $\mathbb{GH}_P$  module $U$ and a weight $\mu \otimes \lambda$ the subspace 
$$\bigoplus_{w\in W_P} U_{w(\lambda)\otimes w(\mu)}$$ 
is a $\mathbb{GH}_P$ submodule of $U$. \end{proof}

\begin{proof}[Proof of (i)]
The simple coroots $\check{\alpha}_1,...,\check{\alpha}_{n}$ will have a dual basis $\beta_1,...,\beta_{n}$ in $\mathfrak{t}_\mathbb{R}$, relative to the Killing form. Let $V$ be an irreducible $\mathbb{GH}$ representation. Let $\mu \otimes \lambda$ be an $\mathscr{A}$ weight of $V$ which is maximal among $Re(\lambda)$. Let $\Pi_P = F = F(Re(\lambda))$. Let $\mathfrak{a}_s$ (respectively $\mathfrak{a}_s^*$) be the elements of $\mathfrak{t}$ (respectively $\mathfrak{t}^*$) perpendicular to $\mathfrak{a}_P^*$ (respectively $\mathfrak{a}_P$). The space $\mathfrak{t}^*$ splits
$$\mathfrak{t}^* = \mathfrak{a}_P^* \oplus \mathfrak{a}_s^*.$$ 
We can restrict characters of $\mathfrak{t}$ to $\mathfrak{a}_P$ (respectively $\mathfrak{a}_s$) by considering the projection of the character in $\mathfrak{t}^*$ to $\mathfrak{a}_P^*$ (respectively $\mathfrak{a}_s^*$).

Let $\nu = \mu \otimes \lambda |_{\mathfrak{a}_P}$. Since $\lambda$ was considered maximal then by construction $\nu \in \mathfrak{a}_P^{*^+}$. Let $U$ be an irreducible representation of $\mathbb{GH}_P$ appearing in $V$ such that $S(\mathfrak{a}_P)$ acts by $\nu$. Let $\mu \otimes \phi $ be a $\mathbb{C}T \otimes S(\mathfrak{a}_s)$ weight of $U$. Since $\mathbb{C}T \otimes S(\mathfrak{t}) \cong \mathbb{C}T \otimes S(\mathfrak{a}_s) \otimes S(\mathfrak{a}_P)$ then
 $\mu \otimes \phi \otimes \nu = \mu \otimes (\phi + \nu)$ is a $\mathbb{C}T \otimes S(\mathfrak{a})$ weight of $V$. 
 $$Re(\phi + \nu ) = \sum_{j\notin F} c_j \beta_j - \sum_{j\in F} z_i \check{\alpha}_i, c_j > 0,$$
 while
 $$Re(\lambda ) = \sum_{j\notin F} c_j \beta_j - \sum_{j\in F} d_i \check{\alpha}_i, c_j > 0, d_i \geq 0.$$
 To prove $U$ is a tempered representation of $\mathbb{GH}_{P_s}$ it is sufficient to prove that $z_i \geq 0$.

 Let $F_2 = \{i \in F : z_i < 0 \}$ and $F_1 = F - F_2$. Then $Re(\phi + \nu) \geq \sum_{j \notin F} c_j \beta_j - \sum_{i \in F_1} z_i \check{\alpha}_i$. Thus by Lemma \ref{lang2} $Re(\phi + \nu)_0 \geq \sum_{j \notin F} c_j \beta_j = Re(\lambda)_0$. But $Re(\lambda) \geq Re(\phi + \nu)$, hence $Re(\lambda)_0 = Re(\phi + \nu)_0$ and therefore $F(Re(\lambda)_0) = F(Re(\phi +\lambda)_0)$. Thus  $\phi + \nu$ is in $S_F$ and $z_i \geq 0$ for all $i$. 
 The inclusion of $\mathbb{GH}_P$ modules $U \subset V$ induces a nonzero map $\pi: \mathbb{GH} \otimes_{\mathbb{GH}_P} U \to V$ given by $\pi(h \otimes w) = h.w$. Since $V$ is irreducible, $V$ is a quotient of $\mathbb{GH} \otimes_{\mathbb{GH}_P} U$. \end{proof}

This argument is very similar to the argument given by Evens for the case of graded Hecke algebras. 
Note that this argument implies that every weight $\mu' \otimes \lambda'$ of $U$ has $F(Re(\lambda')) = F$.

\begin{proof}[Proof of (ii)]

The space $U$ is naturally embedded in $\mathbb{GH}\otimes_{\mathbb{GH}_P} U$. $U$ is a $\mathbb{GH}_P$ module therefore it is invariant under $W_P$. Lemma \ref{orbit} implies that the weights of $\mathbb{GH}\otimes_{\mathbb{GH}_P} U$ are $w(\mu)\otimes w(\lambda)$ where $w \in W$ and $\mu \otimes \lambda$ is a weight of $U$.  Considering the weights of $(\mathbb{GH}\otimes_{\mathbb{GH}_P} U)/U$, these are $w(\mu) \otimes w(\lambda)$ where $w \neq 1$ and is a coset representative of $W /W_P$, alternatively $w \in W^P = \{w\in W : w(R_P^+) \subset R^+\}$. Note that for all $w \in W$ one can write $w$ as a product of $w^P \in W^P$ and $w_P \in W_P$. 

Let $\mu \otimes \lambda $ be a weight of $U$, and write 
$$Re(\lambda) =  \sum_{j\notin F} c_j \beta_j - \sum_{j\in F} d_i \check{\alpha}_i, c_j > 0, d_i \geq 0.$$ Then if $w \in W^P$,
$Re(w\lambda) = \sum_{j\notin F} c_j w\beta_j - \sum_{j\in F} d_i w \check{\alpha}_i.$
Define $\rho: \mathfrak{t} \to \mathbb{C}$ by $\rho(\check{\alpha}) = 1, \check{\alpha} \in \check\Pi$. Since $w: \Pi_P \to R^+$ then $\rho(w(\check{\alpha}_i)) \geq \rho (\check{\alpha}_i)$, for $ i \in F$. Since $\beta_j$ is a fundamental weight, $w(\beta_j) \leq \beta_j$, 
with equality if and only if each expression of $w$ as a product of simple reflections is such that each simple reflection fixes $\beta_j$. If we make this requirement for all $j \notin F$ then this implies $w \in W_P$ hence  $w \in W_P \cap W^P = \{1\}$, therefore $w = 1$. Thus we can assume that if $w \in W^P \setminus 1$ then $\rho(Re(w(\lambda))) < \rho( Re(\lambda))$. 

Fix a weight $\mu \otimes \lambda$ such that $\rho(Re(\lambda))$ is maximal, then $\mu \otimes \lambda$ can not occur as a weight of $(\mathbb{GH}\otimes_{\mathbb{GH}_P} U)/U$. 
This implies that if a submodule $Z$ of $\mathbb{GH}\otimes_{\mathbb{GH}_P} U$ contains $\mu \otimes \lambda$ then $Z$ contains $U$ and hence is $\mathbb{GH}\otimes_{\mathbb{GH}_P} U$. 
Define $I_{max}$ to be the sum of all submodules of $\mathbb{GH}\otimes_{\mathbb{GH}_P} U$ which do not contain $\mu\otimes \lambda$ then $I_{max}$ is maximal and $(\mathbb{GH}\otimes_{\mathbb{GH}_P} U)/I_{max}$ is the unique irreducible quotient.

\end{proof}

\begin{proof}[Proof of (iii)]
Suppose $\pi: J(P,U) \cong J(P',U')$. Let $\mu \otimes \lambda$ (respectively $\mu' \otimes \lambda'$) be a weight of $U$ (respectively $U'$) which is maximal with respect to $\rho$. Suppose $F(Re(\lambda)) \neq F(Re(\lambda'))$. Then it follows that $\mu\otimes\lambda$ is not a weight of $U'$ and $\mu'\otimes\lambda'$ is not a weight of $U$. Therefore $\mu \otimes \lambda$ is a weight of $(\mathbb{GH}\otimes_{\mathbb{GH}_P} U')/U'$ which suggests that $\rho(Re(\lambda)) < \rho(Re(\lambda'))$. However exchanging $\lambda$ with $\lambda'$ suggests $\rho(Re(\lambda')) < \rho(Re(\lambda))$, which can not be the case. Hence $F(Re(\lambda)) = F(Re(\lambda'))$ and $P = P'$.

Since $J(P,U) \cong J(P',U')$ is irreducible Lemma \ref{orbit} implies there exists a $w \in W$ such that $w(\mu)\otimes w(\lambda) = \mu'\otimes \lambda'$.  
If we suppose that $w \notin W_P$ then $w$ has part of its decomposition in $w^P$, by the proof of (ii) this suggests that 
$$\rho(Re(w(\lambda))) = \rho(Re(\lambda')) < \rho(Re(\lambda)).$$
However if this is the case then $\lambda$ is not maximal with respect to $\rho$. Therefore $w(\mu) = \mu'$ where $w \in W_P$. $\pi(U) = U'$ since $U$ (respectively $U'$) is the unique $\mathbb{GH}_P$ submodule which has a weight $\mu_1\otimes \lambda_1$ such that $\rho(Re(\lambda)) = \rho(Re(\lambda_1))$ and $\mu$ is in the same $W_P$ orbit as $\mu_1$. Similarly $\rho(Re(\lambda')) = \rho(Re(\lambda_1))$ and $\mu'$ is in the same $W_P$ orbit as $\mu_1$.
Hence $U \cong U'$ as $\mathbb{GH}_P$ submodules. 
\end{proof}

\end{subsection}

\end{section}

\begin{section}{Dunkl-Opdam subalgebra}\label{DOsubalg}
In this section we study the Dunkl-Opdam subalgebra, $\mathbb{H}_{DO}$ defined in \cite{DO03}. The algebra $\mathbb{H}_{DO}$ is a subalgebra of the rational Cherednik algebra associated to the complex reflection group $G(m,p,n)$. This subalgebra exists for any parameter $\underline{t}$ and its existence is independent of the parameters ${\textbf{c}}_1,...,\textbf{c}_{m-1}$. We show that $\mathbb{H}_{DO}$ is a naturally occurring example of a non-faithful Drinfeld algebra.  Section \ref{Dirac} endows $\mathbb{H}_{DO}$ with a Dirac operator. From the defining presentation given by \cite{DO03} this subalgebra is a generalized Hecke algebra. Therefore we have a Langlands classification for $\mathbb{H}_{DO}$. This sets up Section \ref{reptheory} in which we describe the representation theory of $\mathbb{H}_{DO}$ as blocks corresponding to multipartitions and representations of the graded Hecke algebra of type A. 

\begin{subsection}{The rational Cherednik algebra}\label{rationalcherednikdef}
Dunkl and Opdam \cite{DO03} introduced  the rational Cherednik algebra.
Let $G \subset GL(V)$ be a complex reflection group with reflections $\mathcal{S}$.  Let $\langle , \rangle$ be the natural pairing of $V$ and $V^*$. Let $\alpha_s \in V$  be a $\lambda$ eigenvector for $s \in \mathcal{S}$ and let $v_s\in V^*$ be a $\lambda^{-1}$ eigenvectors  for $s\in \mathcal{S}$ such that $\lambda \neq 1$, and $\langle \alpha_s ,v_s \rangle = 1$. For every reflection $s\in \mathcal{S}$ introduce the parameters $\underline{t}$, $\textbf{c}_s\in \mathbb{C}$  such that $\textbf{c}_s = \textbf{c}_{s'}$ if $s'$ and $s$ are in the same conjugacy class.
The rational Cherednik algebra is defined as the quotient of the associative $\mathbb{C}$ algebra 
$$T(V \oplus V^*) \ltimes \mathbb{C}[G]$$
by the relations
$$[x,x'] = [y,y'] = 0, \hspace{1cm} \text{ for all } x, x' \in V, y ,y' \in V^*,$$
$$[x,y] = \underline{t} \langle x ,y \rangle - \sum_{s\in \mathcal{S}} \textbf{c}_s \frac{\langle \alpha_s, y\rangle\langle  x , v_s\rangle}{\langle \alpha_s, v_s\rangle}s \hspace{1cm} \forall x \in V, y\in V^*,$$
$$g^{-1} v g = g(v), \hspace{1cm} \forall v \in V \oplus V^*.$$

If one restricts to rational Cherednik algebras associated to classical complex groups $G(m,p,n)$ then \cite{DO03} show there is a set of commuting operators inside the rational Cherednik algebra. The main part of these operators is quadratic in a special basis of $V$ and $V^*$.
We give a particular presentation of the rational Cherednik algebra associated to $G(m,1,n)$.

Define a generating set for $G(m,1,n)$ consisting of the reflections $\{s_{i,i+1}: i =1,...,n-1\}$ in $S_n$  and the reflections $\{g_i:i=1,...,n\}$ which have order $m$, we may write $s_i$ for $s_{i,i+1}$. Let $\eta$ be a primitive $m^{th}$ root of unity. Given $G(m,1,n)$ acting on $V$ let $x_i \in V$ be the vectors such that $w (x_i) = x_{w(i)}$  for $w \in S_n$ and 
$$g_i (x_j) = \begin{cases} \eta^{-1} x_j & \text{ if } i = j, \\ x_j & \text{ otherwise.}\end{cases}$$
Let $\{y_1,...,y_n\} \in V^*$ be the dual basis to $\{x_1,...,x_n\}$.
For $G = G(m,1,n)$ there are $m+1$ conjugacy classes of reflections, for reflections in the conjugacy class of $s_{1,2}$ then let $\textbf{k}\in \mathbb{C}$ denote their parameter. Similarly for reflection conjugate to $g_1^l$ denote the parameter by $\textbf{c}_l\in \mathbb{C}$.
\begin{definition} The rational Cherednik algebra for $G(m,1,n)$ and parameters $\textbf{k},\textbf{c}_l,\underline{t}$, $\mathcal{H}_t(G(m,1,n))$ is the quotient of the $\mathbb{C}$ algebra $T(V\oplus V^*) \ltimes \mathbb{C}[G(m,1,n)]$ by the relations
$$[x_i,x_j] = [y_i,y_j] = 0, $$
$$[x_i,y_i] = \underline{t}  - \textbf{k} \sum_{l=1}^{m-1}\sum_{i \neq j} s_{i,j} g_i^{-l}g_j^l - \sum_{l=1}^{m-1}\underline{c}_lg_i^l,$$
$$[x_i,y_j] =  \textbf{k} \sum_{l=1}^{m-1} s_{i,j} g_i^{-l}g_j^l,$$
$$g^{-1} v g = g(v), \text{ for all } v \in V \text{ or } V^*.$$

\end{definition}

\begin{subsection}{Dunkl-Opdam quadratic operators} 

\begin{definition} For $ i \leq n$ define elements $\mathcal{H}_t(G(m,1,n))$
$$z_i = y_ix_i+  \textbf{k} \sum_{l=1}^{m-1}\sum_{i > j} s_{i,j} g_i^{-l}g_j^l +\frac{1}{2} \sum_{l=1}^{m-1}\textbf{c}_l g_i^l +\frac{1}{2}\underline{t},$$
$$=  x_iy_i  -\textbf{k} \sum_{l=1}^{m-1}\sum_{i < j} s_{i,j} g_i^{-l}g_j^l-\frac{1}{2}\sum_{l=1}^{m-1}\textbf{c}_l g_i^l-\frac{1}{2}\underline{t}.$$
\end{definition}

These operators were defined in \cite{DO03} and also appeared in \cite{G10} for the specialization at $\underline{t} = 1$. Martino \cite{M14} used them for general $\underline{t}$ to study the blocks of the rational Cherednik algebra. 

\begin{definition} The Dunkl-Opdam subalgebra $\mathbb{H}_{DO}(G(m,1,n))$  of the rational Cherednik algebra is the subalgebra generated by $G(m,1,n)$ and $z_i$ for $i=1,...n$.\end{definition}

 \begin{remark} The following relations hold in $\mathbb{H}_{DO}(G(m,1,n)$\label{relationssymmetric}
$$[z_i,z_j] = 0 \hspace{1cm} \text{ for } i,j = 1,...,n,$$
$$[z_i,g_k] = 0, \hspace{1cm} \forall i,k = 1,...n,$$
$$[z_j,s_{i,i+1}] = 0 \hspace{1cm} \text{ for } j \neq i,i+1,$$
$$z_i s_{i,i+1} = s_{i, i+1} z_{i+1} -\textbf{k}\epsilon_{ij}.$$
Here $\epsilon_{ij} = \sum_{l=1}^{m-1} g_i^l g_j^{-l}.$\end{remark}

In fact $\mathbb{H}_{DO}(G(m,1,n))$ is isomorphic to the $\mathbb{C}[\textbf{k}]$ associative algebra generated by $z_i$ and $G(m,1,n)$ subject to the relations stated in Remark \ref{relationssymmetric}.
\end{subsection}

\end{subsection}
\begin{subsection}{Dunkl-Opdam subalgebra admits a non-faithful Drinfeld presentation}
In this section we derive a new presentation of $\mathbb{H}_{DO}$ which demonstrates that $\mathbb{H}_{DO}$ is a non-faithful Drinfeld algebra. Thus simultaneously showing that one can associate a Drinfeld algebra to $G(m,1,n)$, also giving a natural example of a non-faithful Drinfeld algebra.

We introduce Jucys-Murphy elements for $G(m,1,n)$. These are well known, however the tool that we use here is that we consider two different sets of Jucys-Murphy elements. 

\begin{definition} We define Jucys-Murphy elements for $G(m,1,n)$. 
$$M_i = \sum_{k<i} \sum_{s=0}^{m-1} s_{k,i} g_k^{-s}g_i^s,$$
$$\underline{M_i} = \sum_{k>i} \sum_{s=0}^{m-1} s_{i,k} g_i^{-s}g_k^s.$$
\end{definition}The commutator $[M_i,M_j] = 0 = [\underline{M_i},\underline{M}_j]$ by a standard argument using the fact that $\sum_{j\leq i}M_i$ is in the centralizer of the subgroup generated by $\{s_{k-1,k}, g_k : k \leq i\}$.  It should be noted that $[M_i, \underline{M_j}] \neq 0$ for $i >j$.
Furthermore $$s_i M_i = M_{i+1}s_i -  \sum_{s=0}^{m-1} g_i^{-s}g_{i+1}^s$$ and 
$$s_i \underline{M}_i = \underline{M}_{i+1}s_i +  \sum_{s=0}^{m-1} g_i^{-s}g_{i+1}^s.$$

If we adjust $z_i$ by $-\textbf{k}M_i$ or $\textbf{k}\underline{M}_i$,  $\hat{z}_i = z_i -\textbf{k} M_i$ ($\underline{\hat{z}}_i = z_i +\textbf{k} \underline{M_i}$ respectively)  the elements $\hat{z}_i$ satisfy the relations $s_i \hat{z}_i s_i = \hat{z}_{i+1}$ and $g_j\hat{z}_ig_j^{-1} = \hat{z}_i$. Hence we obtain an action of $G(m,1,n)$ on the set $\hat{z}_i$. However the set $\{\hat{z}_i\}$ no longer commutes. This presentation was given in \cite[Corollary 3.6]{DO03} where the symbol $T_ix_i$ denotes $\hat{z}_i$. We provide an exposition of the presentation with $\{\underline{\hat{z}}_i\}$ using a simple automorphism of $\mathbb{H}_{DO}$.

\begin{corollary}\cite[Corollary 3.6]{DO03}\label{2ndpres} Let $\hat{z}_i = z_i -\textbf{k} M_i$ then $\hat{z}_i$ and $G$ generate $\mathbb{H}_{DO}$ and the following relations define $\mathbb{H}_{DO}(G(m,1,n)$:
$$s_j\hat{z}_i = \hat{z}_{s_j(i)}s_j,$$
$$[g_i,\hat{z}_j] = 0,$$
$$[\hat{z}_i , \hat{z}_j] = \textbf{k}(\hat{z}_i - \hat{z}_j) \sum_{s=0}^{m-1} s_{s_{i,j}}g_i^{-s}g_j^s.$$
\end{corollary}

\begin{lemma} Let $\Phi : \mathbb{H}_{DO} \to \mathbb{H}_{DO}$ such that
$$\Phi(z_i) = -z_{n+1-i},$$
$$\Phi(s_i) = s_{n-i},$$
$$\Phi (g_i) = g_{n+1-i}.$$
The map $\Phi$ is an automorphism of $\mathbb{H}_{DO}$. 
Furthermore $\Phi (M_i) = \underline{M}_{n+1-i}$.  \end{lemma} 
\begin{proof} 
$$\Phi (s_i z_i - z_{i+1}s_i -\textbf{k}\epsilon_{i,i+1}) = -s_{n-i}z_{n+1-i} + z_{n-i}s_{n-i} -\textbf{k}\epsilon{n-i,n+1-i}.$$ We formally define $\Phi$ as a map from $\mathbb{C}G \rtimes S(V) \to \mathbb{H}_{DO}$ then $\Phi$ takes the set of defining relations in $\mathbb{H}_{DO}$ to itself. $\Phi$ is surjective since it takes generators to generators. Hence we can define $\Phi$ as a automorphism on $\mathbb{H}_{DO}$. \end{proof}

Using $\Phi$ we define the presentation of $\mathbb{H}_{DO}$ with generators $\{\underline{\hat{z}}_i\}$.
\begin{lemma} \label{3rdpres} Let $\underline{\hat{z}}_i = z_i +\textbf{k} \underline{M_i}$, then the set $\underline{\hat{z}}_i$ and $G$ generate $\mathbb{H}_{DO}$. Further the following relations hold: 
$$[\underline{\hat{z}}_i , \underline{\hat{z}}_j] = -\textbf{k}(\underline{\hat{y}}_i - \underline{\hat{z}}_j) \sum_{s=0}^{m-1} s_{s_{i,j}}g_i^{-s}g_j^s,$$
$$s_i \underline{\hat{z}}_i = \underline{\hat{z}}_{i+1} s_i,$$
$$[s_i,\underline{\hat{z}}_j]=0, \forall j \neq i ,i+1,$$
$$[g_i,\underline{\hat{z}}_j] = 0.$$\end{lemma}
\begin{proof}

 $\Phi(\hat{z}_i) = \Phi(z_i -\textbf{k} M_i) =- z_{n+1-i} -\textbf{k} \underline{M}_{n+1-i} = - \underline{\hat{z}}_{n+1-i}.$
Hence $\underline{\hat{z}}_i$ and $G$ generate $\mathbb{H}_{DO}$ since they are images of a generating set under the automorphism $\Phi$. From the definition of $\Phi$,
$$[\underline{\hat{z}}_i,\underline{\hat{z}}_j] = \Phi([-\hat{z}_{n+1-i},-\hat{z}_{n_1-j}]) = \Phi ([\hat{z}_{n+1-i},\hat{z}_{n+1-j}])$$
$$= \Phi(\hat{z}_{n+1-i} - \hat{z}_{n+1-j}) \sum_{s=0}^{m-1} s_{(n+1-i,n+1-j)}g_{n+1-i}^{-s}g_{n+1-j}^s) = -(\underline{\hat{z}}_i -\underline{\hat{z}}_j)\sum_{s=0}^{m-1} s_{s_{i,j}}g_i^{-s}g_j^s).$$
Similarly
$$s_i\underline{\hat{z}}_i =- \Phi (s_{(n+1-i,n-i)}\hat{z}_{n+1-i}) = -\Phi (\hat{z}_{n-i} s_{(n+1-i,n-i)})$$
$$= \underline{\hat{z}_{i+1}} s_i.$$
The relations we gave are images of relations in the second presentation under $\Phi$. Furthermore, the generators and relations are exactly the images of the generators and relations of a presentation hence \ref{3rdpres} gives another presentation of $\mathbb{H}_{DO}$.\end{proof}

We will now work towards a fourth presentation of $\mathbb{H}_{DO}$. We are aiming for a Drinfeld presentation of $\mathbb{H}_{DO}$ which must have the commutator $[\hat{z}_i,\hat{z}_j]$ be an element of the group algebra. 
We now work through several Lemmas to prove that we can alter $z_i$ by $\frac{\textbf{k}}{2} (M_i -\underline{M}_i)$ to give the Drinfeld presentation we expect. 

We observe that the commutators of the set $\{\hat{z}_i\}$, and similarly $\{\underline{\hat{z}}_j\}$ can be expressed as commutators of the Jucys-Murphy elements and $z_i$.
\begin{lemma}\label{commutatorofzhat} The commutator of the operators $\hat{z}_i$ and $\underline{\hat{z}}_i$ are such that:
$$[\hat{z}_i,\hat{z}_j] = \textbf{k}([z_j,M_i]- [z_i,M_j]).$$
Similarly for $\underline{\hat{z}}_i$ and $\underline{M_i}$.
$$[\underline{\hat{z}}_i,\underline{\hat{z}}_j] = \textbf{k}([z_i, \underline{M}_j]-[z_j,\underline{M}_i]).$$

\end{lemma}

\begin{proof} 
$$[\hat{z}_i,\hat{z}_j] = [z_i -\textbf{k} M_i, z_i -\textbf{k} M_j ] = [z_i,z_j] -\textbf{k} [z_i,M_j] +\textbf{k} [z_j, M_i] +\textbf{k}^2 [M_i,M_j],$$
$$= \textbf{k}([z_j,M_i]- [z_i,M_j]),$$
since $[M_i,M_j] = [\underline{M}_i.\underline{M}_j] = [z_i,z_j] = 0.$

\end{proof}

\begin{lemma} For operators $\hat{z}_i$ and $\underline{\hat{z}_j}$, 
$$[\hat{z}_i,\hat{z}_j] + [\underline{\hat{z}}_i,\underline{\hat{z}}_j] \in \mathbb{C}G.$$
\end{lemma} 
\begin{proof} 
Using Corollary \ref{2ndpres} and Lemma \ref{3rdpres} we can expand the commutators:
$$[\hat{z}_i,\hat{z}_j] + [\underline{\hat{z}}_i,\underline{\hat{z}}_j]$$

$$= (\hat{z}_i - \hat{z}_j) \sum_{s=0}^{m-1} s_{i,j}g_i^{-s}g_j^s -(\underline{\hat{z}}_i - \underline{\hat{z}}_j) \sum_{s=0}^{m-1} s_{s_{i,j}}g_i^{-s}g_j^s.$$
Writing out $\hat{z}_i$ and $\underline{\hat{z}}_i$ in terms of the commuting operators $z_i$ one obtains
$$= (z_i - \textbf{k}M_i - z_j +\textbf{k} M_j) \sum_{s=0}^{m-1} s_{i,j}g_i^{-s}g_j^s -(z_i +\textbf{k}\underline{M_i} - z_j -\textbf{k}\underline{M}_j) \sum_{s=0}^{m-1} s_{s_{i,j}}g_i^{-s}g_j^s$$
Cancelling out the operators $z_i$ we arrive at the element of the group algebra 
$$ [\hat{z}_i,\hat{z}_j] + [\underline{\hat{z}}_i,\underline{\hat{z}}_j]= \textbf{k}(M_j + \underline{M}_j - M_i - \underline{M}_i)  \sum_{s=0}^{m-1} s_{i,j}g_i^{-s}g_j^s \in \mathbb{C}G.$$
\end{proof}

\begin{lemma}\label{commutant} The commutators of $z_i - \frac{\textbf{k}}{2} M_i + \frac{\textbf{k}}{2}\underline{M}_i$ are in $\mathbb{C}G$. $$[z_i - \frac{\textbf{k}}{2} M_i + \frac{\textbf{k}}{2}\underline{M}_i, z_j - \frac{\textbf{k}}{2}M_j + \frac{\textbf{k}}{2}\underline{M}_i] \in \mathbb{C}G.$$
\end{lemma}

\begin{proof}Expanding out the commutator linearly: $$[z_i - \frac{\textbf{k}}{2} M_i + \frac{\textbf{k}}{2}\underline{M}_i, z_j - \frac{\textbf{k}}{2}M_j + \frac{\textbf{k}}{2}\underline{M}_j]$$
$$= [z_i, z_j]  +\frac{\textbf{k}}{2} \left( [z_j,M_i] - [z_i,M_j]    + [z_i,\underline{M}_j] - [z_j,\underline{M}_i] \right) $$
$$+\left(\frac{\textbf{k}}{2}\right)^2 \left([M_i, M_j] - [M_i, \underline{M}_j] - [\underline{M}_i,M_j] +[\underline{M}_i,\underline{M}_j] \right)$$
Using Lemma \ref{commutatorofzhat}:
$$= \frac{1}{2} \left([\hat{z}_i,\hat{z}_j] + [\underline{\hat{z}}_i,\underline{\hat{z}}_j] \right) - \left(\frac{\textbf{k}}{2}\right)^2\left( [M_i, \underline{M}_j] +[\underline{M}_i,M_j] \right) \in \mathbb{C}G,$$
$$= \frac{\textbf{k}}{2} \left( M_j + \underline{M}_j - M_i - \underline{M}_i\right) \sum_{s=0}^{m-1} s_{s_{i,j}}g_i^{-s} g_j^s - \left(\frac{\textbf{k}}{2}\right)^2\left( [M_i, \underline{M}_j] +[\underline{M}_i,M_j] \right) \in \mathbb{C}G.$$
\end{proof}

\begin{lemma}\label{Gaction} $s_j(z_i - \frac{\textbf{k}}{2} M_i + \frac{\textbf{k}}{2}\underline{M}_i)s_j^{-1} = (z_{s_j(i)} - \frac{\textbf{k}}{2} M_{s_j(i)} + \frac{\textbf{k}}{2}\underline{M}_{s_j(i)})$. \end{lemma}
\begin{proof} The result follows from compiling these three relations: 
$$s_i z_i = z_{i+1}s_i + \textbf{k}\epsilon_{i,i+1},$$
$$ s_i M_i = M_{i+1}s_i + \epsilon_{i,i+1},$$
$$s_i \underline{M}_i  = \underline{M}_{i+1}s_i -\epsilon_{i,i+1}.$$

\end{proof}

\begin{theorem}\label{Drinfeldpresentation} There exists a  presentation of $\mathbb{H}_{DO}$ given by elements $\{\tilde{z}_i: i=1,...,n\}$ and elements in $G$ such that:
$$s_i \tilde{z}_j s_i^{-1} = s_i(\tilde{z}_j) ,$$
$$g_i \tilde{z}_j = \tilde{z}_j g_i \hspace{1cm} \forall i,j = 1,...,n,$$
$$ [\tilde{z}_i,\tilde{z}_j] \in \mathbb{C}G.$$
\end{theorem} 
\begin{proof} Let $\tilde{z}_i = \frac{1}{2}( \hat{z}_i + \underline{\hat{z}}_j) = z_i-\frac{\textbf{k}}{2}(M_i - \underline{M}_i)$ then the first two relations follows from Lemma \ref{Gaction} and by Lemma \ref{commutant} their commutant is in $\mathbb{C}G$. 
One may be worried that we have defined an algebra that surjects onto $\mathbb{H}_{DO}$ but does not procure an injection. However performing the above arguments in reverse setting $z_i = \tilde{z}_i +\frac{\textbf{k}}{2}(M_i - \underline{M}_i)$ shows that the original relations follow from these relations. \end{proof}
\begin{definition}\label{skewformsdef}
Give $V$ a basis $\{v_i\}$ and recall that $S_n$ act on this basis by permutations.
Let $\theta$ be the homomorphism of $G(m,1,n)$ onto $S_n$.
$(V,\phi)$ is the standard representation of $S_n$, now define $(V,\rho)$ to be the representation of $G$ via the projection onto $S_n$, that is, 
$\rho: G \to GL(V)$ via $\rho(g) =  \phi(\theta(g)).$  We define skew-symmetric forms on $V$ for elements in $G(m,1,n)$:
$$b_{s_{ij}s_{jk}}(v_p,v_q) =\underline{k}^2\left(\langle \epsilon_i-\epsilon_j,\epsilon_p\rangle \langle \epsilon_j-\epsilon_k,\epsilon_q\rangle- \langle \epsilon_i-\epsilon_j,\epsilon_q\rangle \langle \epsilon_j-\epsilon_k,\epsilon_p\rangle\right),\text{ for } 0<i<j<k\leq n,$$
$$b_{s_{ij}g_i^lg_{j}^{-l}s_{jk}g_j^{l'}lg_{k}^{-l'}} = b_{s_{ij}s_{jk}} \text{ for all }l,l' = 0,...,m-1$$
$$b_g = 0 \text{ otherwise.}$$
\end{definition}
\begin{theorem} The algebra $\mathbb{H}_{DO}$ is a Drinfeld algebra. More concretely $\mathbb{H}_{DO}$ is isomorphic to $\mathbb{C}G \rtimes T(V)$ with the relations: 
$$[u,v] = \sum_{g\in G} b_g(u,v) g \hspace{1cm} \forall u,v \in V,$$
where $b_g$ are skew-symmetric forms on $V$ defined in Definition \ref{skewformsdef}.
\end{theorem}
\begin{proof}Conjugating $b_g$ by $g_i$ must fix $b_g$ since $g_i$ acts trivially on $V$. Quotienting  $\mathbb{H}_{DO}(G(m,1,n)$ by $g_i -1$ gives a quotient isomorphic to the graded Hecke algebra of type A $\mathbb{H}(S_n)$. Hence the forms $b_g$ must agree with, under the quotient the forms that construct the Drinfeld presentation of the graded Hecke algebra for $S_n$. The forms $b_{s_{ij}s_{jk}}$ descend to the forms, labelled the same element, defining the graded Hecke algebra as a Drinfeld algebra. Conjugating by various $g_i^l$ gives the forms $b_{s_{ij}g_i^lg_{j}^{-l}s_{jk}g_j^{l'}lg_{k}^{-l'}}$ above. Since $b_1 = 0$ in $\mathbb{H}(S_n)$ then $b_k = 0$ for all $k \in \ker \rho$. There are no other elements of $G(m,1,n)$ such that $\dim V^\rho(g) = \dim V - 2$, therefore the rest of the $b_g =0$ for all $g$ not mentioned above. \end{proof}

\end{subsection}
\begin{subsection}{Dunkl-Opdam subalgebra is a generalised graded Hecke algebra}

Recall Definition \ref{GGHdef} of the generalised graded Hecke algebra associated to the root system of $W$, $T$ and parameter function $\tilde{c}$.
Let $\epsilon_{ij} = \sum_{l=1}^{m-1} g_i^l g_j^{-l}\in \mathbb{C}G(m,1,n).$
The algebra $\mathbb{H}_{DO}$ is isomorphic to, as a vector space, $\mathbb{C}[G(m,1,n)] \otimes S(V)$, with multiplication such that $\mathbb{C}[G(m,1,n)]$ and $S(V)$ are subalgebras and the following cross relations hold;
$$[z_i,g_k] = 0, \hspace{1cm} \forall i,k = 1,...n,$$
$$[z_j,s_{i,i+1}] = 0 \hspace{1cm} \text{ for } j \neq i,i+1,$$
$$z_i s_{i,i+1} = (i, i+1) z_{i+1} - \textbf{k}\epsilon_{ij}.$$

If we substitute $G(m,1,n)\cong S_n \rtimes (\mathbb{Z}/m\mathbb{Z})^n$ for $W\rtimes T$ then $\mathbb{H}_{DO}$ is a generalised Hecke algebra with parameter function
$$\tilde{c}(s_{ij}) =\underline{ \textbf{k}}\epsilon_{ij}\in \mathbb{C}[(\mathbb{Z}_m)^n].$$

Since $\mathbb{H}_{DO}$ is a generalised graded Hecke algebras we can apply the Langlands classification from Section \ref{Langlandssection}. Therefore we can construct every representation of $\mathbb{H}_{DO}$ as a quotient of the module inducted from a tempered module of a parabolic subalgebra. 
\begin{corollary} Let $U= \check{U} \otimes \mathbb{C}_\nu$ be such that $\check{U}$ is a tempered $\mathbb{H}_{DO}$ module and $\nu$ is a character of $\mathfrak{a}^{*^+}$. Every irreducible representation of $\mathbb{H}_{DO}$ can be constructed as a quotient of a tempered module of a parabolic subalgebra $\mathbb{H}_P$. That is it is a quotient of 
$$\mathbb{H}_{DO}\otimes_{ \mathbb{H}_P}U.$$

\end{corollary}

\end{subsection}

\begin{section}{Constructing the representations of $\mathbb{H}(G(m,1,n))$ from $\mathbb{H}(S_n)$.}\label{reptheory}
In this section we prove the representations of $\mathbb{H}(G(m,1,n))$ can be built up from blocks of irreducible representations of the graded Hecke algebras associated to the symmetric group. This is very similar to how one can build the representations of $W(B_n)$ from the pullback of two representations of symmetric groups, $S_a$ and $S_b$, where $a+b = n$

We denote the usual graded Hecke algebra of type $A_{k-1}$ by $\mathbb{H}(S_k)$, $\eta$ denotes a fixed primitive $m^{th}$ root of unity.  

We define $\mathbb{N}$ to include zero. Let $A \subset \mathbb{N}^m$ be the set of vectors such that the coordinates sum to $n$. Then let $\underline{a} = (a_0,...,a_{m-1})$ be a vector in $A$, explicitly $\sum_{i=0}^{m-1} a_i = n$.  We define the character $\mu_{\underline{a}} \in \mathbb{C}[T]^*$ by 
$$\mu_{\underline{a}} (g_j) = \eta^i$$ 
where $\sum_{k=0}^{i-1}a_{k}<j \leq \sum_{k=0}^{i}a_{k}.$

This character takes the first $a_0$ reflections to $1$ it then takes the following $a_1$ reflections to $\eta$ then the following $a_2$ to $\eta^2$ and continues in this way. Finally it takes the last $a_{m-1}$ reflections to $\eta^{m-1}$. 
The set $A$ will become a parametrising set.

\begin{example} Let $n = 5$ and $m=3$, define $\omega$ to be a primitive $3rd$ root of unity. The character of $\mathbb{C}[(\mathbb{Z}/3)^5]$ associated to the vector $(1,1,3)$ is such that;
$$\mu_{(1,1,3)} (g_1) = 1, \hspace{1cm} \mu_{(1,1,3)} (g_2) = \omega,$$
$$\mu_{(1,1,3)} (g_3) = \mu_{(1,1,2)} (g_4) = \mu_{(1,1,3)} (g_5)=\omega^2.$$
\end{example}

 If we take the $S_n$ orbits of $\mathbb{C}[T]^*$ then a representative set of these orbits is $$\{\mu_{\underline{a}}| \underline{a} \in A \}.$$

Let $\irr(\mathbb{H}(G(m,1,n))$ be the set of isomorphism classes of irreducible modules for $\mathbb{H}(G(m,1,n))$. We define $\irr(\mathbb{H}(G(m,1,n)) | \mu_{\underline{a}})$ to be the subset of $\irr(\mathbb{H}(G(m,1,n))$ consisting of representations that have a weight $\mu_{\underline{a}}\otimes \lambda$ for any $\lambda \in S(V)^*$. Similarly, we will denote the set of irreducible representations of a complex algebra $B$ by $\irr (B)$.

\begin{lemma} The irreducible representations of $\mathbb{H}(G(m,1,n))$ split into disjoint sets labelled by $A$,
$$\irr(\mathbb{H}(G(m,1,n))) = \bigsqcup_{\underline{a}\in A} \irr(\mathbb{H}(G(m,1,n))| \mu_{\underline{a}}).$$
\end{lemma}

\begin{proof} Since every irreducible representation of $\mathbb{H}(G(m,1,n))$ has at least one $\mathbb{C}[T]$ weight then by Lemma \ref{orbit} it must contain one and only one $S_n$ orbit, hence it must contain exactly one $\mu_{\underline{a}}$. Therefore every irreducible representation $V$ is in exactly one of the sets $\irr(\mathbb{H}(G(m,1,n))| \mu_{\underline{a}}).$
\end{proof}

Let $S_{a_0} \times S_{a_1} \times ... S_{a_{m-1}}$ be the parabolic subgroup of $S_n$ generated by $s_{\alpha}$ for 
$$\Pi_{\underline{a}}  = \{\epsilon_i - \epsilon_{i+1} |\sum_{k=0}^{j-1}a_{k} \leq i < \sum_{k=0}^{j}a_{k} \text{ for some } j \}\subset\Pi.$$
Fix $\underline{a} \in A$. The stabiliser, $\stab(\mu_{\underline{a}}) \subset \mathbb{H}(G(m,1,n))$, of the character $\mu_{\underline{a}}$ is generated by $\mathbb{C}[T]$, $S(V)$ and $s_i \in S_{a_0} \times S_{a_1} \times ... S_{a_{m-1}} \subset S_n$. 
$\Pi_{\underline{a}}$ is equivalent to the set of simple roots $\epsilon_i - \epsilon_{i+1}$ such that $\mu_{\underline{a}}(g_i) = \mu_{\underline{a}}(g_{i+1})$.
This is the parabolic subalgebra associated to the subset $\Pi_{\underline{a}} \subset \Pi$ defined in  Definition $\ref{parabolic}$.
\begin{lemma}\label{otimes} The subalgebra $\stab(\mu_{\underline{a}})$ which stabilises the character $\mu_{\underline{a}}$ is isomorphic to $\mathbb{H}(G(a_0,1,m)) \otimes\mathbb{H}(G(a_1,1,m)) \otimes ... \otimes\mathbb{H}(G(a_{m-1},1,m))$. \end{lemma}
\begin{proof}
The subalgebra generated by $S_{a_0} \times S_{a_1} \times ... S_{a_{m-1}}$, $\mathbb{C}[T]$ and $S(V)$ certainly contains $\mathbb{H}(G(a_i,1,m))$ for every $i = 0,...,m-1$. The algebra $\mathbb{H}(G(a_i,1,m))$ consists, as a vector space of $S(V_i) \otimes \mathbb{C}[T_i]\otimes S_{a_i}$ where $V_i$ is the span of $\epsilon_j$, and $\mathbb{C}[T_i]$ is generated by $g_j$ such that $\sum_{k=0}^{i-1}a_{k}<j \leq \sum_{k=0}^{i}a_{k}$. We have $V_0 \oplus ...\oplus V_{m-1} = V$ hence $S(V_0) \otimes ...\otimes S(V_{m-1}) = S(V_0 \oplus ...\oplus V_{m-1}) = S(V)$. Similarly $\mathbb{C}[T_0]\otimes ...\otimes \mathbb{C}[T_{m-1}] = \mathbb{C}[T]$, and $\mathbb{C}[S_{a_0}]\otimes...\otimes \mathbb{C}[S_{a_{m-1}}] = \mathbb{C}[S_{a_0}\times ... \times S_{a_{m-1}}]$. Hence as a vector space:
$$\stab(\mu_{\underline{a}}) = \bigotimes_{i=0}^{m-1} S(V_i)\otimes \mathbb{C}[T_i]\otimes \mathbb{C}[S_{a_i}].$$
Each $\mathbb{H}(G(a_i,1,m))$ is a subalgebra and as vector spaces we have equality, one just needs to check that each subalgebra commutes with the other subalgebras. We already know $S_{a_i}$ and $S_{a_j}$ commute, for $i \neq j$, and $S_{a_i}$ commutes with $\mathbb{C}[T_j]$ because $S_{a_i}$ fixes $T_j$. Similarly $S_{a_i}$ fixes $V_j$ so  $s_{\alpha_i} \in S_{a_i}$ commutes with $\epsilon_j \in S(V_j)$.\end{proof}

\begin{lemma}\label{functors}
The set of irreducible representations $\irr(\mathbb{H}(G(m,1,n))| \mu_{\underline{a}})$ is in natural one-one correspondence with $\irr(\stab(\underline{a}))| \mu_{\underline{a}})$.
The bijection $F$ is defined by
$$F^{-1}:\irr(\stab(\underline{a})| \mu_{\underline{a}})\to \irr(\mathbb{H}(G(m,1,n))| \mu_{\underline{a}}),$$
$$F^{-1}(W) =   Ind_{\stab(\underline{a})}^{\mathbb{H}(G(m,1,n))}W$$
and
$$F:\irr(\mathbb{H}(G(m,1,n)) | \mu_{\underline{a}}) \to \irr(\stab(\underline{a})| \mu_{\underline{a}})$$
$$F(U) =  \text{Unique irreducible submodule of } Res^{\mathbb{H}(G(m,1,n))}_{\stab(\underline{a})}U \text{ with weight } \mu_{\underline{a}}.$$
\end{lemma}
For an irreducible module $U$ in $\irr(\mathbb{H}(G(m,1,n))| \mu_{\underline{a}})$,  $F^{-1}(U)$ is the $\mu_{\underline{a}}$-weight space of $U$.
\begin{proof} Let $P$ be the corresponding partition defined by the subset $\Pi_{\underline{a}} \subset \Pi$ and
set $C$ to be the set of coset representatives of the parabolic group $S_P$ in $S_n$. For a $\stab(\underline{a})$ module $W$ the module  $W^{c}$, for $c \in C$, is isomorphic to $W$ as a vector space with the action $$b \cdot W^{c} = c^{-1} b c W \text{ for all } b \in \stab(\underline{a}).$$ 
We must check that $F(W)$ for $W \in \irr(\stab(\mu_{\underline{a}}))$ is an irreducible $\mathbb{H}(G(m,1,n))$ module. The rest follows easily.
As a vector space $\mathbb{H}(G(m,1,n)) \cong  \oplus_{c\in C} c^{-1}\stab(\underline{a})c$. Here $C$ is a set of coset representative of the parabolic group $S_P$ in $S_n$.
The $\stab(\underline{a})$-composition series of $Ind_{\stab(\underline{a})}^{\mathbb{H}(G(m,1,n))} W$ consists of the $\stab(\underline{a})$ modules $W^{c}$ where $c\in C$. We must show that $Ind_{\stab(\underline{a})}^{\mathbb{H}(G(m,1,n))} W$ is irreducible. The module $W$ is an irreducible $\stab(\underline{a})$ module. If $W$ is an irreducible $G(m,1,a_0) \times .. \times G(m,1,a_{m-1})$ module then utilising Mackey's criterion
 for finite groups we need to show that$W^c$ are non isomorphic.  However by construction $W$ has only weights containing $\mu_{\underline{a}}$, and $W^{c}$ will only have weights containing $c(\mu_{\underline{a}})$, and since $S_P$ is the stabiliser of $\mu_{\underline{a}}$ in $S_n$ then for all $c \neq 1$ we have  $c(\mu_{\underline{a}}) \neq \mu_{\underline{a}}$. Therefore each $W^{c}$ has a different set of weights and hence are not isomorphic. Hence if $W$ is an irreducible $G(m,1,n)$ module then  using Mackey's irreducibility criterion $Ind_{\stab(\underline{a})}^{\mathbb{H}(G(m,1,n))} W$ is an irreducible $G(m.1,n)$ module and hence $F(W)$ is irreducible as a $\mathbb{H}(G(m,1,n)$ module.

If $W$ is reducible as a $G(m,1,a_0) \times .. \times G(m,1,a_{m-1})$ module then $W = \bigoplus V_i$ as irreducible  $G(m,1,a_0) \times .. \times G(m,1,a_{m-1}$ modules. by the same argument as above the induction of each of these is an irreducible $G(m,1,n)$ module. We have
$$Ind_{\stab(\underline{a})}^{\mathbb{H}(G(m,1,n))} W = \bigoplus Ind^{G(m,1,n)} V_i$$ as a $G(m,1,n)$ module. Suppose that $Ind_{\stab(\underline{a})}^{\mathbb{H}(G(m,1,n))} W$ is not irreducible as a $\mathbb{H}(G(m.1,n)$ module then some direct sum of $V_i's$ is a submodule. Suppose $\bigoplus_{i\in I} Ind V_i$ is a submodule. Therefore $\bigoplus{i \in I} V_i$ is a $\stab(\underline{a})$ submodule of $W$ hence since $W$ is irreducible as a $\stab(\underline{a})$ module  then $I = {0,1,...,m-1}$ and the only irreducible non-trivial submodule is the whole module. Therefore $Ind_{\stab(\underline{a})}^{\mathbb{H}(G(m,1,n))} W$ is irreducible.

 However by construction $W$ has only weights containing $\mu_{\underline{a}}$, and $W^{c}$ will only have weights containing $c(\mu_{\underline{a}})$, and since $S_P$ is the stabiliser of $\mu_{\underline{a}}$ in $S_n$ then for all $c \neq 1$ we have  $c(\mu_{\underline{a}}) \neq \mu_{\underline{a}}$. Therefore each $W^{c}$ has a different set of weights and hence are not isomorphic. So using Mackey's irreducibility criterion $F(W)$ is irreducible. 
It is easy to verify that $F^{-1} \cdot F (V) = V$ using the universal property of induced modules and similarly $F \cdot F^{-1} (W) = W$. 
\end{proof}

Given a representation of $(V,\pi) \in  \irr(\stab(\underline{a})| \mu_{\underline{a}})$ we can explicitly describe how $g_i \in G(m,1,n)$ acts. Since this algebra stabilises $\mu_{\underline{a}}$ this is the only $\mathbb{C}[T]$ weight occurring in $V$. Therefore 
$$g_i = \mu_{\underline{a}}(g_i) Id.$$

Let $\alpha_i = \epsilon_i - \epsilon_{i+1}$, if we study the relation $s_{\alpha_i} \alpha_i = s_{\alpha_i}(\alpha_i)s_{\alpha_i} + \sum_{l=0}^{m-1}g_i^lg_{i+1}^{-l} $ in $\mathbb{H}(G(m,1,n)$, on $(V,\pi)$, the element  $\sum_{l=0}^{m-1}g_i^lg_{i+1}^{-l}$ is equal to $\pi(\sum_{l=0}^{m-1}g_i^lg_{i+1}^{-l}) = \sum_{l=0}^{m-1} \mu_{\underline{a}}(g_i)^l \mu_{\underline{a}}(g_{i+1})^{-l}Id$ which then equals $m$ if $\mu_{\underline{a}}$ is constant on $g_i$ and $g_{i+1}$ and $\sum_{l=0}^{m-1}g_i^lg_{i+1}^{-l}$ is zero if  $\mu_{\underline{a}} (g_i) \neq \mu_{\underline{a}}(g_{i+1})$.
One can summarise, on any representation in $\irr(\stab(\underline{a})| \mu_{\underline{a}})$
$$\sum_{l=0}^{m-1}g_i^lg_{i+1}^{-l} = \begin{cases} m & \text{ if } s_{\epsilon_i - \epsilon_{i-1}} \in  \stab(\mu_{\underline{a}}), \\
0 & \text{ otherwise.} \end{cases}$$
Recall from Lemma \ref{otimes} that $\stab(\underline{a})$ is isomorphic to $\mathbb{H}(G(a_0,1,m))\otimes ... \otimes\mathbb{H}(G(a_{m-1},1,m))$. We have shown that if  $(V,\pi) \in \irr(\stab(\underline{a}| \mu_{\underline{a}})$ then the relations  $s_{\alpha_i} \alpha_i = s_{\alpha_i}(\alpha_i)s_{\alpha_i} +\textbf{k} \sum_{l=0}^{m-1}g_i^lg_{i+1}^{-l} $  become $s_{\alpha_i} \alpha_i = s_{\alpha_i}(\alpha_i)s_{\alpha_i} + m $ on $V$ for all $\alpha_i \in \Pi_{\underline{a}}$ and the relations $s_{\alpha_i} \alpha_i = s_{\alpha_i}(\alpha_i)s_{\alpha_i} +\textbf{k} \sum_{l=0}^{m-1}g_i^lg_{i+1}^{-l} $  become $s_{\alpha_i} \alpha_i = s_{\alpha_i}(\alpha_i)s_{\alpha_i} + 0 $ on $V$ for all $\alpha_i \notin \Pi_{\underline{a}}$ .  Hence we can conclude that if $(V,\pi)  \in  \irr(\stab(\underline{a}| \mu_{\underline{a}}))$ then this representation factors through the algebra 
$$\mathbb{H}(S_{a_0})\otimes ... \otimes\mathbb{H}(S_{a_{m-1}}),$$
via the quotient by the ideal $I_{\underline{a}} = <g_i - \mu_{\underline{a}}(g_i)Id>$.
Explicitly 
$$\pi: \mathbb{H}(G(a_0,1,m))\otimes ... \otimes\mathbb{H}(G(a_{m-1},1,m))\twoheadrightarrow \mathbb{H}(S_{a_0})\otimes ... \otimes\mathbb{H}(S_{a_{m-1}}) \to GL(V).$$

Where $\mathbb{H}(S_n)$ is the usual graded Hecke algebra associated to $S_n$, with parameter $c(\alpha) = m\textbf{k}$.

\begin{lemma}\label{quotientspullback} The set $ \irr(\mathbb{H}(G(a_0,1,m))\otimes ... \otimes\mathbb{H}(G(a_{m-1},1,m))| \mu_{\underline{a}})$ is in one-one correspondence with $\irr(\mathbb{H}(S_{a_0}))\otimes ... \otimes \irr(\mathbb{H}(S_{a_{m-1}}))$. \end{lemma}
\begin{proof}
The irreducible representation in  $\irr(\mathbb{H}(G(a_0,1,m))\otimes ... \otimes\mathbb{H}(G(a_{m-1},1,m))| \mu_{\underline{a}})$ all occur as pullbacks of the irreducible representations of $\mathbb{H}(S_{a_0})\otimes ... \otimes\mathbb{H}(S_{a_{m-1}})$ via the specific quotient of $\mathbb{H}(G(a_0,1,m))\otimes ... \otimes\mathbb{H}(G(a_{m-1},1,m))| \mu_{\underline{a}}$ onto  $\mathbb{H}(S_{a_0})\otimes ... \otimes\mathbb{H}(S_{a_{m-1}})$, with the ideal 
$$I_{\underline{a}} = <g_i - \mu_{\underline{a}}(g_i)Id | i =1,...,n>.$$
Furthermore given a representation $U$ of $\mathbb{H}(S_{a_0})\otimes ... \otimes\mathbb{H}(S_{a_{m-1}})$ one can create a representation in $ \irr(\mathbb{H}(G(a_0,1,m))\otimes ... \otimes\mathbb{H}(G(a_{m-1},1,m))| \mu_{\underline{a}})$ by pulling back the representation $U$ from the quotient of $I_{\underline{a}}$. 
\end{proof}

\begin{theorem} \label{moritaequiv}
The irreducible representations of $\mathbb{H}(G(m,1,n))$ split into blocks which are induced from products of $\mathbb{H}(S_a)$ representations:
$$\irr(\mathbb{H}(G(m,1,n))) \cong \bigsqcup_{\underline{a}\in A} \irr(\mathbb{H}(S_{a_0}))\otimes ... \otimes \irr(\mathbb{H}(S_{a_{m-1}})).$$

\end{theorem} 

If one considers a tempered  $\mathbb{H}(G(m,1,n))$ module, that is $V$ such  that the $\mathbb{C}[T
]\otimes S(\mathfrak{t})$ weights, $\mu \otimes \lambda$ are such that $Re(\lambda(x_i)) \leq 0$ for all fundamental coweights and that $Re(\lambda|_{\mathfrak{a}_\mathbb{R}}) = 0$. This condition is only dependent on the $S(\mathfrak{t})$ weight $\lambda$ therefore a tempered $\mathbb{H}(G(m,1,n)$ correspond to a $\mathbb{H}(S_n)$ tempered module with weight $\lambda$. Hence every $\mathbb{H}(S_n)$ tempered module $V$ there will correspond to $m$ different $\mathbb{H}(G(m,1,n)$ tempered modules. Each tempered $\mathbb{H}(G(m,1,n))$ module will be a pullback of the module $V$. However the difference between the $m$ different modules are that the short reflections $g_i$ will act by $\eta^{j}$ for fixed $j = 1,...,m$. 

This gives a method to parametrise the Langlands data for an irreducible $\mathbb{H}(G(m,1,n))$ module via tempered modules of $\mathbb{H}(S_a)$. 
Recall that every irreducible $\mathbb{H}(G(m,1,n))$ module can be realized as a quotient of 
$$\mathbb{H}(G(m,1,n)) \otimes_{\mathbb{H}(G(m,1,n))_P} \check{U}\otimes \mathbb{C}_\nu.$$
If we fix an irreducible module $X$ then using the above realization, we associate to it Langlands data $(P, U)$. 

Fix $P= (p_0,...,p_{m-1})$ a partition of n with at most $m$ parts. 
The tempered $\mathbb{H}(G(m,1,n))$ modules are the pullbacks of tempered $\mathbb{H}(S_{a_0})\otimes ... \otimes \mathbb{H}(S_{a_{m-1}})$ modules. 

Recall that the size of a partition $\lambda = \{x_1,...,x_j\}$ is $\sum_{i=1}^j x_j$. The tempered modules of graded Hecke algebra with real central character correspond to partitions $(e,\phi)$ where $\phi$ is nilpotent \cite[3.3]{CT13},\cite{KL87}. In the case of $W= S_n$ $e$ is always $1$ and $phi$ is characterized by it's Jordan form and hence corresponds to a partition of $n$. Hence the tempered modules of $\mathbb{H}(S_{p_0})\otimes ... \otimes \mathbb{H}(S_{p_{m-1}})$ with real central character will correspond to a set of $m$ partitions $\{\lambda_0,...,\lambda_{m-1}\}$ such that the sum of the sizes of the partitions $\lambda_i$ is $a_i$.

\begin{theorem} Let $P=(a_0,...a_{m-1})$ be a set associated to $\underline{a} = (a_0,...,a_m)$ with at most $m$ parts. We associate to $P$, a parabolic subalgebra $\mathbb{H}_P(G(m.1,n)) \subset \mathbb{H}(G(m,1,n))$. The tempered modules of the parabolic algebra $\mathbb{H}_P(G(m,1,n))$ are built up from tempered modules of each parabolic part $\mathbb{H}(G(m,1,a_i))$. By above a tempered module of  $\mathbb{H}(G(m,1,a_i))$ corresponds to a tempered module of $\mathbb{H}(S_{a_i})$. 
The tempered modules of  $\mathbb{H}(S_{a_i})$ with real central character are labelled by partitions of $a_i$. Hence tempered module of  $\mathbb{H}_P(G(m,1,n))$ with real central character  are labelled by multipartitions $\{\lambda_0,...,\lambda_{m-1}\}$ with $m$ partitions such that the size of $\lambda_i$ equals $a_i$. Furthermore one can construct these tempered modules via the pullback of

$$\mathbb{H}(G(m,1,n)) \to_{\phi} \mathbb{H}(S_{a_0}))\otimes ... \otimes (\mathbb{H}(S_{a_{m-1}}) \to GL( V_{\lambda_0} \otimes ...\otimes V_{\lambda_{m-1}}).$$
Where $V_{\lambda_i}$ is the tempered module of $\mathbb{H}(S_{a_i})$ corresponding to the partition $\lambda_i$ and $\phi$ is the quotient by the ideal $I_{\underline{a}} = <g_i - \mu_{\underline{a}}(g_i)Id | i =1,...,n>$.
\end{theorem}

\end{section}
\end{section}

\begin{section}{Dirac cohomology of the Dunkl-Opdam subalgebra}\label{diraccohofDO}

In this section we will use the description of irreducible representations from Section \ref{reptheory} to describe how the Dirac operator for the Dunkl-Opdam subalgebra acts on irreducible modules. We will show that  the Dirac operator $\mathcal{D}_{DO}$ for $\mathbb{H}_{DO}$ descends to a relevant Dirac operator for a tensor of type A graded Hecke algebras and then describe the Dirac operator in terms of Dirac operators for type A parabolic algebras.

Let $A$ be a Drinfeld algebra $T(V) \rtimes \mathbb{C}[G]/R$. We have an associated Clifford algebra $C(V)$, with respect to the $G$-invariant symmetric product $<,>$ . Given a $G$-invariant basis $B$ we defined the Dirac operator to be $$\sum_{b\in B} b \otimes b^* \in A \otimes C(V).$$

We have two presentations of the Dunkl-Opdam subalgebra, one producing  the Lusztig presentation  \ref{relationssymmetric} with commuting basis elements and the Drinfeld presentation used in Theorem \ref{Drinfeldpresentation} which shows that $\mathbb{H}(G(m,1,n))$ is a Drinfeld algebra. We used the Lusztig presentation to show the Morita equivalence of the Dunkl-Opdam subalgebra to a sum of tensors of type A graded Hecke algebras, this uses parabolic sub algebras. However the Dirac theory developed for the Dunkl-Opdam subalgebra uses the Drinfeld presentation. This Drinfeld presentation does not admit parabolic subalgebras. 

Let us recall that to transform between from the classical presentation to the Drinfeld presentation one takes the standard basis $\{z_1,...,z_n\}$ of the reflection representation of $S_n$ which along with $G(m,1,n)$ gives the classical presentation. Then to obtain the Drinfeld presentation we use the generators: $$\tilde{z}_i =z_i 
+\frac{\textbf{k}}{2}\sum_ {i<j}s_{i,j}\sum_{l=1}^{m-1}g_i^{-l}g_j^l -\frac{\textbf{k}}{2}\sum_{j<i} s_{i,j}\sum_{l=1}^{m-1}g_i^{-l}g_j^l  =z_i +\frac{\textbf{k}}{2}( \underline{M}_i-M_i ).$$ Recall $M_i$ and $\underline{M}_i$ are Jucys-Murphy elements of $G(m,1,n)$ with reverse orderings.
 The Dirac element in terms of the Drinfeld presentations is: 
$$\mathcal{D}_{DO} = \sum_{i=1}^n \tilde{z}_i \otimes z_i^*.$$
In terms of the Lusztig presentation $\{z_i\}$ the Dirac element is 
$$\mathcal{D}_{DO} = \sum_{i=1}^{n}\left( z_i 
+\frac{\textbf{k}}{2}\sum_ {i<j}s_{i,j}\sum_{l=1}^{m-1}g_i^{-l}g_j^l -\frac{\textbf{k}}{2}\sum_{j<i} s_{i,j}\sum_{l=1}^{m-1}g_i^{-l}g_j^l \right) \otimes z_i^*.$$
\begin{definition} Given a  $\mathbb{H}$ module $X$ and a spinor $S$ of $C(V)$, then $\mathcal{D}_{DO}: X \otimes S \to X \otimes S$. The Dirac cohomology of $X$ with respect to $S$ is defined to be 
$$\Ker(\mathcal{D}_{DO} ) / \im(\mathcal{D}_{DO}) \cap \Ker (\mathcal{D}_{DO}).$$
Since $\mathcal{D}_{DO}$ $\sgn$-commutes with the group $\widetilde{G}$ then the Dirac cohomology is naturally a $\widetilde{G}$ module.\end{definition}

From section \ref{reptheory} if $V$ contains a weight $\mu_{\underline{a}}$ corresponding to $\underline{a} = \{a_0,...,a_{m-1}\}$ then we will write $V_{\mu_{\underline{a}}}$ for the $\mathbb{C}[T]$-weight space corresponding to the weight $\underline{a}$. We can decompose $V$ into $\mathbb{C}[T]$ weight spaces.
$$V = \bigoplus_{c\in C} V_{\mu_{\underline{a}}}^c.$$ 
Lemma \ref{functors} shows that $V_{\mu_{\underline{a}}}$ is the image of the functor $F$. It is a $\stab(\mu_{\underline{a}})$ module and is the pullback of a tensor of $\mathbb{H}_{s_{a_i}}$ modules.
A problem that occurs is that the Dirac operator $\mathcal{D}_{DO}$ does not sit in the subalgebra $\stab(\mu_{\underline{a}}) \cong \bigotimes_{i=1}^{m-1} \mathbb{H}_{DO}(G(m,1,a_i)$. 
We will look at the Dirac operators already given for the standard type A graded Hecke algebra.
\begin{definition}\label{DfortypeA}\cite{BCT12} For the graded Hecke algebra $\mathbb{H}(S_k)$ the Dirac operator is 
$$D_{S_k}=\sum_{i=1,...,k}  \left( z_i +\frac{m\textbf{k}}{2}\sum_ {i<j}s_{i,j} -\frac{m\textbf{k}}{2}\sum_{j<i} s_{i,j} \right) \otimes z_i^*.$$\end{definition}
\begin{remark}We abuse notation here as $z_i$ in this context denotes the same basis as we have used in the definition of $\mathbb{ H}_{DO}$ but of course it is not in the same algebra. We justify this since all surjections of $\mathbb{H}_{DO}$ onto $\mathbb{H}_{S_n}$ preserve this notation. We have used the parameter $m\textbf{k}$ as opposed to $\textbf{k}$ for $\mathbb{H}_{S_n}$ because naturally our map sends $\mathbb{H}_{DO}$ to $\mathbb{H}_{S_n}$ with parameter $m\textbf{k}$.
\end{remark}

Recall that the weight space  $V_{\mu_{\underline{a}}}$ is naturally a $\bigotimes_{i=1}^{m-1} \mathbb{H}_{S_{a_i}}$ module, via the functor $F$ defined in the Lemma \ref{functors}. We extend Definition \ref{DfortypeA} to define Dirac operator for $\bigotimes_{i=1}^{m-1} \mathbb{H}_{S_{a_i}}$:
$$\mathcal{D}_{S_{a_0}\times...\times S_{a_{m-1}}} = \mathcal{D}_{S_{a_0}} \otimes...\otimes \mathcal{D}_{S_{a_{m-1}}}.$$ 
Written out explicitly this is
$$\mathcal{D}_{S_{a_i}\times ...\times S_{a_{m-1}}} = \sum_{i=0}^{m-1}\sum_{j=a_{i-1}}^{j=a_i}\left( z_j +\frac{m\textbf{k}}{2}\sum_ {j<k\leq a_i}s_{j,k} -\frac{m\textbf{k}}{2}\sum_{a_{i-1}<k<j} s_{j,k} \right) \otimes z_j^*.$$
Here we have associated $\bigotimes_{i=1}^{m-1}S(V_i)$ with $S(\oplus V_i)$. Similarly we have substituted $ C(\oplus V_i)= \bigotimes C(V_i)$.
Initially this looks like the Dirac operator for $\mathbb{H}_{S_n}$, however one should notice that not all of the reflections are involved in this Dirac operator. We highlight this with an example.
\begin{example} Let n=3. The Dirac operator for $\mathbb{H}_{S_3}$ is
$$\left(z_1 -\frac{m\textbf{k}}{2}(1,2)-\frac{m\textbf{k}}{2}(1,3)\right)\otimes z_1^* + \left(z_2 +\frac{m\textbf{k}}{2}(1,2) - \frac{m\textbf{k}}{2}(2,3)\right) \otimes z_2^* $$
$$ +\left(z_3 +\frac{m\textbf{k}}{2}(1,3) + \frac{m\textbf{k}}{2}(2,3)\right)\otimes z_3^*.$$
However the Dirac operator for $\mathbb{H}_{s_1}\otimes \mathbb{H}_{S_2} \subset \mathbb{H}_{S_3}$ is
$$z_1\otimes z_1^* + \left (z_2  - \frac{m\textbf{k}}{2}(2,3)\right) \otimes z_2^* + \left(z_3  + \frac{m\textbf{k}}{2}(2,3)\right)\otimes z_3^*.$$
One can see that there are four reflections in the $\mathcal{D}_{\mathbb{H}(S_3)}$ not involved in $\mathcal{D}_{\mathbb{H}(S_1)} \times \mathcal{D}_{\mathbb{H}(S_2)}.$ 
\end{example}

Viewing $V_{\mu_{\underline{a}}}$ as a $\stab(\mu_{\underline{a}})$ module, $V$ is a sum of twists of $F(V)$. Let us look at the $\mathbb{C}[T]$-invariant element of  $\stab(\mu_{\underline{a}})$ which maps to $\mathcal{D}_{S_{a_0}\times...\times S_{a_{m-1}}}$, this is:
$$\sum_{i=0}^{m-1}\sum_{j=a_{i-1}}^{j=a_i}\left( z_j +\frac{\textbf{k}}{2}\sum_{l=1}^{m-1}g_j^{-l}g_k^l\sum_ {j<k\leq a_i}s_{j,k} -\frac{\textbf{k}}{2}\sum_{l=1}^{m-1}g_k^{-l}g_j^l\sum_{a_{i-1}<k<j} s_{j,k} \right) \otimes z_j.$$ 
Written this way one notices that this looks very similar to the $\mathbb{H}_{DO}$ Dirac operator however it excludes the reflections that are not in the parabolic subgroup that stabilises $\mu_{\underline{a}}$. The following lemma shows that the difference vanishes on $V_{\mu_{\underline{a}}}$.
\begin{lemma} Given an irreducible module $V_{\underline{a}}$ with $\mathbb{C}[T]$ weight $\underline{a}$ then on the subspace $F(V_{\underline{a}})$ the Dirac operator for $\mathbb{H}_{DO}$ acts by the Dirac operator $\mathcal{D}_{S_{a_0}\times...\times S_{a_{m-1}}}$. \end{lemma}
\begin{proof} Recall that since $V_{\mu_{\underline{a}}}$ only has one $\mathbb{C}[T]$ weight, namely $\underline{a}$ we can explicitly describe how $\sum_{l=0}^{m-1}g_i^{-l}g_j^l$ acts on this subspace. 
$$\sum_{l=0}^{m-1}g_i^{-l}g_j^l = \begin{cases} m  Id &  \text{ if }\mu_{\underline{a}}(g_i) = \mu_{\underline{a}}(g_i),\\ 0 & \text{ if }\mu_{\underline{a}}(g_i) \neq \mu_{\underline{a}}(g_i).\\\end{cases}$$
This parametrisation of pairs $\{i,j\}$ can be described in another way. If the transposition $s_{i,j}$ stabilises the character $\mu_{\underline{a}}$ then $\sum_{l=0}^{m-1}g_i^{-l}g_j^l = m$. However if $s_{i,j}$ is not in $\stab(\mu_{\underline{a}})$ then $\sum_{l=0}^{m-1}g_i^{-l}g_j^l = 0$ on the $\mu_{\underline{a}}$-weight space.
Ultimately this means that the Dirac operator $\mathcal{D}_{DO} \in \mathbb{H}_{DO} \otimes CL(V)$  preserves the subspace $F^{-1}(V)\otimes \mathcal{S}$ since the transpositions included in $\mathcal{D}_{DO}$ which do not preserve $V_{\mu_{\underline{a}}}$ are preceded by the element  $\sum_{l=0}^{m-1}g_i^{-l}g_j^l$ which acts by zero in this case. Finally since $\mathcal{D}_{DO}$ preserves $V_{\mu_{\underline{a}}}$ it equals an element inside $\stab(\underline{a})\otimes \mathcal{S}$. This is the pull back of $\mathcal{D}_{S_{a_0}\times...\times S_{a_{m-1}}}$ and hence $\mathcal{D}_{DO}$ agrees with $\mathcal{D}_{S_{a_0}\times...\times S_{a_{m-1}}}$ on the $\mu_{\underline{a}}$-weight space of $V$.

\end{proof}

We have described how the Dirac operator acts on the $\underline{a}$ weight space of $V$. Since $\mathcal{D}_{DO}$ is $G(m,1,n)$ invariant we can describe how it acts on the rest of the weight spaces. As discussed in Lemma \ref{functors} the other weight spaces are twists of this space by the coset representatives, $c \in C$ of the parabolic subgroup $S_P$ in  $S_n$. The group $S_P$ fixes the $\underline{a}$ weight space. Therefore if $\mathcal{D}_{DO}$ acts by $\mathcal{D}_{S_{a_0}\times...\times S_{a_{m-1}}}$ on $V_{\mu_{\underline{a}}}$  then $\mathcal{D}$ acts by $c\mathcal{D}_{S_{a_0}\times...\times S_{a_{m-1}}}c^{-1}$ on $cV_{\mu_{\underline{a}}}$. Hence $\Ker (\mathcal{D}_{DO}) \subset V \otimes \mathcal{S}$ is
$$\bigoplus_{c\in C} c \Ker\mathcal{D}_{S_{a_0}\times...\times S_{a_{m-1}}}.$$
Similarly since $\mathcal{D}_{DO}$ acts by $\mathcal{D}_{S_{a_0}\times...\times S_{a_{m-1}}}$ on $\stab(\underline{a})\otimes \mathcal{S}$ then 
$$\im \mathcal{D}_{DO} = \oplus_{c\in C} c \im \mathcal{D}_{S_{a_0}\times...\times S_{a_{m-1}}}.$$ 
We can describe the Dirac cohomology of an irreducible module $X$ in terms of the Dirac cohomology of its corresponding $\mathbb{H}_{S_{a_0}} \otimes...\otimes \mathbb{H}_{S_{a_{m-1}}}$ module.

\begin{theorem} Given an irreducible representation $V$ with $\mathbb{C}[T]$ weight space $V_{\mu_{\underline{a}}}$ then by transforming $V_{\mu_{\underline{a}}}$ to a $\mathbb{H}_{S_{a_0}} \otimes...\otimes \mathbb{H}_{S_{a_{m-1}}}$ module $X_{a_0} \otimes...\otimes X_{a_{m-1}}$ the Dirac cohomology of $V$ is 
$$\bigoplus_{c\in S_P/S_n} c \left (  H_D(X_{a_0}) \otimes ...\otimes H_D(X_{a_{m-1}})\right ),$$
where $H_D(X)$ is the Dirac cohomology of the $\mathbb{H}_{S_k}$-module $X$. \end{theorem}

\begin{definition} Define $\widetilde{F}$ and $\tilde{F}^{-1}$ to be the functors exhibiting the Morita equivalence between $\widetilde{G(m,1,n)}$ and $\bigoplus_{\underline{a}\in A}  \bigoplus_{\underline{a}\in A} \mathbb{C}\widetilde{S_{a_0}} \otimes ...\otimes \mathbb{C}\widetilde{S_{a_{m-1}}}$, similarly to Lemma \ref{functors}.\end{definition}

\begin{corollary} Let $H_D(\bullet)$ denote the functor taking the relevant module to a its Dirac cohomology. 
We have the following commutative diagram:

\begin{centering}

\begin{tikzcd}
\mathbb{H}(G(m,1,n))\text{-mod}\arrow[r, "F"] \arrow[d, "H_D(\bullet)"] & \bigoplus_{\underline{a}\in A}\mathbb{H}_{S_{a_0}} \otimes...\otimes \mathbb{H}_{S_{a_{m-1}}} \text{-mod}\arrow[d, "H_D(\bullet)" ] \\\mathbb{C}\widetilde{G(m,1,n)}\text{-mod} 
& \bigoplus_{\underline{a}\in A} \mathbb{C}\widetilde{S_{a_0}} \otimes ...\otimes \mathbb{C}\widetilde{S_{a_{m-1}}}\text{-mod}\arrow[l, "\widetilde{F^{-1}}"]
\end{tikzcd}
\end{centering}

\end{corollary}

\end{section}

\bibliography{bib}{}

\begin{thebibliography}{10}

\bibitem{BCT12}
D.~Barbasch, D.~Ciubotaru, and P.Trapa.
\newblock Dirac cohomology for graded affine {H}ecke algebras.
\newblock {\em Acta Math.}, 209 (2):197--227, 2012.

\bibitem{BW80}
A.~Borel and N.~Wallach.
\newblock {\em Continuous cohomology, discrete subgroups, and representations
  of reductive groups}.
\newblock Princeton {U}niversity Press, 1980.

\bibitem{C16}
D.~Ciubotaru.
\newblock Dirac cohomology for symplectic reflection algebras.
\newblock {\em Selecta Math.}, 22.1:111--144, 2016.

\bibitem{CT11}
D.~Ciubotaru and P.~E. Trapa.
\newblock Functors for unitary representations of classical real groups and
  affine {H}ecke algebras.
\newblock {\em Advances in Mathematics}, 227(4):1585 -- 1611, 2011.

\bibitem{CT13}
D.~M. Ciubotaru and P.~E. Trapa.
\newblock Characters of {S}pringer representations on elliptic conjugacy
  classes.
\newblock {\em Duke Mathematical Journal}, 162(2):201--223, 2013.

\bibitem{D06}
C.~Dez{\'e}l{\'e}e.
\newblock Generalized graded {H}ecke algebras of types {B} and {D}.
\newblock {\em Comm. Algebra}, 34(6):2105--2128, 2006.

\bibitem{D86}
V.~G. Drinfel'd.
\newblock Degenerate affine {H}ecke algebras and {Y}angians.
\newblock {\em Functional Analysis and Its Applications}, 20(1):58--60, 1986.

\bibitem{DO03}
C.~F. Dunkl and E.~M. Opdam.
\newblock Dunkl operators for complex reflection groups.
\newblock {\em Proceedings of the London Mathematical Society}, 86(1):70--108,
  2003.

\bibitem{E96}
S.~Evens.
\newblock The {L}anglands classification for graded {H}ecke algebras.
\newblock {\em Proceedings of the AMS}, 124(4):1285--1290, 1996.

\bibitem{G10}
S.~Griffeth.
\newblock Towards a combinatorial representation theory for the rational
  {C}herednik algebra of type {$G(r, p, n)$}.
\newblock {\em Proceedings of the Edinburgh Mathematical Society},
  53(2):419–445, 2010.

\bibitem{HP06}
J.~S. Huang and P.~Pandzic.
\newblock {\em Dirac operators in representation theory}.
\newblock Mathematics: Theory and {A}pplications. Birk{\"a}user {B}oston, {MA},
  2006.

\bibitem{KL87}
D.~Kazhdan and G.~Lusztig.
\newblock Proof of the {D}eligne-{L}anglands conjecture for {H}ecke algebras.
\newblock {\em Inventiones mathematicae}, 87(1):153--215, 1987.

\bibitem{M14}
M.~Martino.
\newblock Blocks of restricted rational {C}herednik algebras for {$G(m,d,n)$}.
\newblock {\em Journal of Algebra}, 397:209 -- 224, 2014.

\bibitem{M13}
E.~Meinrenken.
\newblock {\em Clifford algebra and {L}ie theory}.
\newblock Results in {M}athematics and {R}elated {A}reas. 3rd {S}eries. {A}
  {S}eries of {M}odern {S}urveys in {M}athematics. Springer, {H}eidelberg,
  2013.

\bibitem{RS03}
A.~Ram and A.~Shepler.
\newblock Classification of graded {H}ecke algebras for complex reflection
  groups.
\newblock {\em Comment. Math. Helv.}, 78(2):308--334, 2003.

\end{thebibliography}
\bibliographystyle{abbrv}
\end{document}